\documentclass[12pt]{amsart}
\usepackage[normalem]{ulem}
\usepackage{amsmath}
\usepackage{hyperref}
\usepackage{amssymb}
\usepackage{mathrsfs}
\usepackage{tikz}
\tikzstyle{vertex}=[circle]
\tikzstyle{goto}=[->,shorten >=1pt,>=stealth,semithick]
\usetikzlibrary{graphs}

\allowdisplaybreaks[2]

\newcommand{\Ef}{\widehat E_0}
\newcommand{\Eu}{\widehat E_\infty}
\newcommand{\Et}{\widehat E_{\text{tight}}}
\newcommand\Iso{\operatorname{Iso}}

\newcommand\go{\G^{(0)}}

\newcommand\cs{\ensuremath{C^{*}}}
\newcommand\cls{\overline{\supp}}
\newcommand{\inv}{^{-1}}

\newcommand{\Z}{\mathbb{Z}}
\newcommand{\KK}{\mathbb{K}}
\newcommand{\sgx}{S_{G,X}}
\newcommand{\ggx}{\G_{G,X}}
\newcommand{\ogx}{\mathcal{O}_{G,X}}
\newcommand{\N}{\mathbb{N}}
\newcommand{\ew}{\varnothing}

\newcommand{\fs}{\subseteq_{\text{fin}}}
\newcommand {\OGE }{{\mathcal {O}_{G,E}}}

\newcommand {\SGE }{{S_{G,E}}}

\newcommand {\CG }{{\mathcal {G}_{G,E}}}
\newcommand{\vp}{\varphi}
\newcommand{\gzeA}{\mathcal{G}_{\Z,E_A}}

\newcommand{\intiso}{\operatorname{Iso}(\G)^\circ}

\newcommand{\G}{\mathcal{G}}

\newcommand{\C}{\mathbb{C}}
\DeclareMathOperator{\supp}{supp}
\DeclareMathOperator{\lsp}{span}
\DeclareMathOperator{\clsp}{\overline{\lsp}}

\newcommand{\D}{{\mathcal{D}}}
\newcommand{\s}{{\mathcal{S}}}

\newcommand{\gt}{\G_{\text{tight}}}

\newtheorem{thm}{Theorem}[section]

\newtheorem{lemma}[thm]{Lemma}
\newtheorem{prop}[thm]{Proposition}
\newtheorem{cor}[thm]{Corollary}

\theoremstyle{definition}
\newtheorem{definition}[thm]{Definition}

\theoremstyle{remark}
\newtheorem{remark}[thm]{Remark}

\newtheorem{noname}[equation]{}
\numberwithin{equation}{section}

\newcommand{\CC}{\mathbb{C}}

\newcommand{\Bb}{\mathcal{B}}
\newcommand{\Cc}{\mathcal{C}}
\newcommand{\Hh}{\mathcal{H}}

\begin{document}

\title[Simplicity of groupoid algebras]{Simplicity of algebras associated to non-Hausdorff groupoids}

\author {Lisa Orloff Clark}
\address {Lisa Orloff Clark\\School of Mathematics and Statistics, Victoria University of Wellington, PO Box 600, Wellington 6140, New Zealand}
\email {lisa.clark@vuw.ac.nz}

\author {Ruy Exel}
\address {Ruy Exel\\Departamento de Matem\'atica, Universidade Federal de Santa Catarina, 88040-970 Florian\'opolis SC, Brazil}
\email {exel@mtm.ufsc.br}\urladdr {http://www.mtm.ufsc.br/~exel/}

\author {Enrique Pardo}
\address {Enrique Pardo\\Departamento de Matem\'aticas, Facultad de Ciencias\\ Universidad de C\'adiz, Campus de
Puerto Real\\ 11510 Puerto Real (C\'adiz)\\ Spain.} \email {enrique.pardo@uca.es}\urladdr
{https://sites.google.com/a/gm.uca.es/enrique-pardo-s-home-page/}

\author{Aidan Sims}
\address{Aidan Sims\\School of Mathematics and Applied Statistics, University of Wollongong, Wollongong NSW 2522, Australia.} \email{asims@uow.edu.au}

\author{Charles Starling}
\address{Charles Starling\\Carleton University, School of Mathematics and Statistics, 4302 Herzberg Laboratories, 1125 Colonel By Drive, Ottawa, ON, Canada, K1S 5B6.}
\email{cstar@math.carleton.ca}\urladdr{https://carleton.ca/math/people/charles-starling/}

\dedicatory {To Frederick Noel Starling}

\thanks {The authors thank the anonymous referee for carefully reading our paper and for the helpful comments.
The first named author was partially supported by Marsden grant from the Royal Society of New
Zealand. The second named author was partially supported by CNPq. The third named author was
partially supported by PAI III grant FQM-298 of the Junta de Andaluc\'{\i }a, and by the
DGI-MINECO and European Regional Development Fund, jointly, through grants MTM2014-53644-P and
MTM2017-83487-P. The fourth named author was partially supported by the Australian Research
Council grant DP150101595. The fifth named author was partially supported by a Carleton University
internal research grant.}



\begin{abstract}

We prove a uniqueness theorem and give a characterization of simplicity for Steinberg algebras
 associated to non-Hausdorff ample groupoids. We also prove a uniqueness theorem and give a
characterization of simplicity for the $\cs$-algebra associated to non-Hausdorff \'etale
groupoids. Then we show how our results apply in the setting of tight representations of inverse
semigroups, groups acting on graphs, and self-similar actions. In particular, we show that the
$\cs$-algebra and the complex Steinberg algebra of the self-similar action of the Grigorchuk
group are simple but the Steinberg algebra with coefficients in $\mathbb{Z}_2$ is not simple.
\end{abstract}

\maketitle

\section{Introduction}\label{sec:intro}
Algebras associated to locally compact groupoids play an important role in both analysis and
algebra. The theory of $\cs$-algebras associated to Hausdorff groupoids, introduced  by Renault in
\cite{Ren},  is fairly well-developed. Connes introduced $\cs$-algebras of non-Hausdorff groupoids
in \cite{Connes82},  but a lot less is known in this setting.  What we do know is  that results
about Hausdorff groupoids often fail when the Hausdorff property is relaxed, see for example
\cite{Exel}.

As much as we might be tempted to treat non-Hausdorff groupoids as pathological outliers, they
appear in crucial examples, see,  for example,  \cite{Connes} and more recently \cite{EP17}. In
this paper we investigate simplicity of algebras associated to non-Hausdorff groupoids.

One important result used to characterize simplicity in the Hausdorff setting is the Cuntz-Krieger
uniqueness theorem.   It gives suitable conditions on the groupoid under which every ideal in the
associated algebra contains a function entirely supported on the unit space.  In this paper we
establish algebraic and analytic Cuntz-Krieger uniqueness theorems for non-Hausdorff groupoid
algebras.

What goes wrong when moving from Hausdorff to non-Hausdorff groupoids? The existence of nonclosed
compact sets wreaks havoc on our understanding of `compactly supported functions'.  We are forced
to consider functions that fail to be continuous and functions whose `open support', that is, the
set of points where the function is nonzero, is not an open set at all.  In fact, there might be
nonzero \emph{singular functions} whose open support has empty interior.

After establishing some preliminaries in Section~\ref{sec:prelim},  we consider the purely
algebraic class of Steinberg algebras in Section~\ref{sec:st}.  Introduced in \cite{St}, Steinberg
algebras
 are built from \emph{ample} groupoids: an ample groupoid is a topological groupoid that has a basis of compact open bisections.
Given an ample groupoid $\G$ with Hausdorff unit space, the complex Steinberg algebra associated
to $\G$ is the convolution algebra consisting of the linear span of characteristic functions (from
$\G$ to $\C$) of  compact open bisections. In the Hausdorff setting, the Steinberg algebra is
simple if and only if the associated groupoid $\G$ is \emph{effective} and \emph{minimal}.  More
generally, we show in Theorem~\ref{thm:SteinbergSimple}  that the Steinberg algebra associated to
$\G$ is simple if and only if $\G$ is minimal, effective, and there are no singular functions.
The proof uses our newly established algebraic Cuntz-Krieger uniqueness theorem,
Theorem~\ref{thm:cku}.

We also present a topological condition on the groupoid that, along with effective and minimal,
implies simplicity.  Given a groupoid $\G$ that is effective and minimal, if every compact open
set of $\G$ is \emph{regular open}, then the associated Steinberg algebra is simple.  Although
this condition is not necessary (see example in subsection~\ref{grig}) it does capture many
examples and is fairly straight forward to check.

We view Steinberg algebras as a laboratory for finding  conditions to characterize
$\cs$-simplicity for groupoid $\cs$-algebras.   We move to the analytic setting in
Section~\ref{sec:an}, where we study  general \emph{\'etale} groupoids, that is, groupoids with a
bases of open bisections (ample groupoids are always \'etale).  Given an \'etale groupoid $\G$, we
can view the elements of the reduced $\cs$-algebra as functions from $\G$ to $\C$ as described in
\cite{KhoshSkand} (see Section~\ref{page_functions}). After establishing a Cuntz-Krieger
uniqueness theorem in Theorem~\ref{thm:reduced uniqueness}, we show in Theorem~\ref{thm:c*simple}
that  the reduced groupoid $\cs$-algebra of $\G$ is simple if and only if $\G$ is minimal,
effective and there are no singular functions in the $\cs$-algebra. We look at the special case of
ample groupoids in Section \ref{csample}.

In Section~\ref{sec:ex} we present broad classes of examples. First, we present a class of
groupoids that are topologically principal and minimal but whose associated algebras are not
simple.   Note that our examples are not effective and demonstrate that, unlike the Hausdorff
situation,  topologically principal does not imply effective.

Next we apply our simplicity results to tight groupoids of inverse semigroup representations, then
specialize to groups acting on  graphs and further specialize to self-similar group actions.
Finally we showcase our results by applying them to the self-similar action of the Grigorchuk
group. We answer the long standing open question of whether or not the algebras associated to this
action are simple.  We show in Theorem~\ref{thm:grigsimple} that both the complex Steinberg
algebra and the $\cs$-algebra are simple. This is a surprise, because \cite[Example~4.5]{Nek16}
(see also our Corollary~\ref{cor:grigZ2notsimple}) shows that the Steinberg algebra with
coefficients in $\mathbb{Z}_2$ is not simple.  Thus the field plays a role in whether or not the
algebra is simple, which is an unexpected non-Hausdorff only phenomenon.

Although we have made significant progress,  here are three open questions that we were unable to
answer:

1.  What are necessary and sufficient conditions on a non-Hausdorff groupoid $\G$ that ensure the
Steinberg algebra is simple?   Although we have necessary and sufficient conditions characterizing
simplicity, one of our conditions is not a groupoid condition, rather, it is a condition about
functions. We still do not know whether minimal and effective alone are necessary and sufficient.

2.  What are necessary and sufficient conditions on a non-Hausdorff groupoid $\G$ that ensure its
reduced $\cs$-algebra algebra is simple?  Again, the best we could do was to impose a condition on
functions.

3. Suppose the complex Steinberg algebra associated to an ample groupoid is simple.  Must the
$\cs$-algebra also be simple?

Regarding~(2) above, a related question is discussed in \cite{ExelPitts}, where the notion of the
gray ideal of a regular inclusion of $\cs$-algebras is defined \cite[Definition 11.8]{ExelPitts}.
Contrasting Lemma~\ref{lem:singiso} and \cite[Proposition 15.3.ii]{ExelPitts}, we see that every
singular element of $\cs_r(\G)$ lies in the gray ideal. So conditions on $\G$ that ensure that the
gray ideal is trivial also rule out the existence of nontrivial singular elements. Under such
conditions, statement~(3) of Theorem~\ref{thm:c*simple} would therefore say that $C^*(\G)$ is
simple whenever $\G$ is minimal and effective.

Interestingly, for twisted groupoid $C^*$-algebras, conditions ensuring that the gray ideal is
trivial must take the twist into account, since there are examples \cite[Section 23]{ExelPitts} in
which the gray ideal vanishes for the twisted groupoid C*-algebra, but not for its untwisted
version.

\section{Preliminaries}
\label{sec:prelim}

\noindent \textbf{Regular open sets}

\medskip

In a topological space we say a subset $B$ is a \emph{regular open} set if $B$ equals the interior
of its closure; that is $B = (\overline{B})^{\circ}.$ The intersection of a collection of regular
open sets is again regular open but the same is not true for unions. See for example
\cite[Chapter~10]{GivHal} for a detailed discussion of regular open sets.

\begin{lemma}
 \label{lem:regopendif}
Let $B$ and $D$ be regular open sets. If $B \setminus D$ is nonempty, then $B \setminus D$ has
nonempty interior.
\end{lemma}

\begin{proof}
  Since $B \setminus D$ is nonempty we have $D \cap B \subsetneq B.$
Regular openness then gives
\[\overline{D \cap B} \subsetneq \overline{B}.\]
 Let $O$ be the complement of $\overline{D \cap B}$.
Thus $O$ is an open set that intersects $\overline{B}$ and hence intersects $B$. Also $O \cap (D
\cap B) = \emptyset$. Thus $\emptyset \neq B \cap O \subseteq B \setminus D$ and hence $B
\setminus D$ has nonempty interior.
\end{proof}

\medskip

\noindent \textbf{\'Etale and ample groupoids}

\medskip

We say a topological groupoid $\G$ is \emph{\'etale} if there is a basis for its topology
consisting of \emph{open bisections}. That is, of open sets $B$ such that the source map
(equivalently the range map) restricts to  a homeomorphism onto an open subset of the unit space
$\go$.   We will always assume our groupoids are \'etale, locally compact and that $\go$ is
Hausdorff in the relative topology.  Since $\G$ is \'etale,  $\go$ is open in $\G$, and $\G$ is
Hausdorff if and only if $\go$ is also closed in $\G$ (see e.g. \cite[Proposition 3.10]{EP16}.) We
will make use of the following lemma often.

\begin{lemma}
\label{lem:lchbis} Let $\G$ be a locally compact, \'etale groupoid such that $\go$ is Hausdorff.
Suppose $B$ is an open bisection. Then $B$ is locally compact and Hausdorff in the relative
topology.
\end{lemma}

\begin{proof}
 Since $B$ is an open bisection, $B$ is homeomorphic to an open subset of the Hausdorff space $\go$.  The lemma follows.
\end{proof}

We say $\G$ is \emph{minimal} if for any $u \in \go$, we have $[u]:= s(r^{-1}(u))$ dense in $\go$.
Equivalently, $\G$ is minimal if for every nonempty open $U \subseteq \go$, we have $[U] :=
s(r^{-1}(U)) = \go$.

For $u \in \go$, denote
\[\G_u := \{\gamma\in \G: s(\gamma) = u\} \quad \text{and } \quad
\G_u^u:= \{\gamma \in \G: s(\gamma) = r(\gamma)=u\}.\]
Then $\G_u^u$ is called the \emph{isotropy group} at $u$; the \emph{isotropy group bundle} of $\G$
is then
\[\operatorname{Iso}(\G):= \bigcup_{u \in \go} \G_u^u.\]
Notice that in an \'etale groupoid $\go$ is open subset of $\operatorname{Iso}(\G)$. If the
interior of $\operatorname{Iso}(\G)$ is equal to $\go$, we say that $\G$ is \emph{effective}.
Thus, if $\G$ is not effective, then there exists an open bisection $B \subseteq \G$ such that $B
\nsubseteq \go$ and for every $\gamma \in B$, $s(\gamma)=r(\gamma)$. In our main results, we
require $\G$ to be second countable. In this setting, if $\G$ is effective, then $\G$ is
\emph{topologically principal}, that is, the collection of units with trivial isotropy group dense
in $\go$ see \cite[Proposition~3.6]{Ren2}. If $\G$ is Hausdorff, the converse is true but in the
non-Hausdorff case the converse does not hold, see Example~\ref{aidanex}.

\medskip

\noindent \textbf{Regular open sets in effective \'etale groupoids}

\medskip

In the following, we demonstrate that effective groupoids contain a lot of regular open sets.

\begin{lemma}\label{lem:go reg open}
Suppose that $\G$ is a locally compact \'etale groupoid. If $\G$ is effective, then $\go$ is
regular open in $\G$.
\end{lemma}
\begin{proof}
Since $r,s$ are continuous, we have $\overline{\go} \subseteq \Iso(\G)$. Hence
$\big(\overline{\go}\big)^\circ \subseteq \Iso(\G)^\circ$. Since $\G$ is effective, we have
$\Iso(\G)^\circ = \go$, and so $\big(\overline{\go}\big)^\circ \subseteq \go$.
\end{proof}

\begin{cor}\label{cor:closed unit set}
Suppose that $\G$ is a locally compact \'etale groupoid. If $\G$ is effective, and $K \subseteq
\go$ is relatively closed in $\go$, then $K^\circ$ is regular open in $\G$.
\end{cor}
\begin{proof}
We have $\overline{K^\circ} \subseteq \overline{\go}$, and so $\big(\overline{K^\circ}\big)^\circ
\subseteq \big(\overline{\go}\big)^\circ = \go$ by Lemma~\ref{lem:go reg open}. Hence
\[
\big(\overline{K^\circ}\big)^\circ \subseteq \overline{K} \cap \go.
\]
Since $K$ is closed in the relative topology in $\go$, we deduce that
$\big(\overline{K^\circ}\big)^\circ \subseteq K$.
\end{proof}

\begin{lemma} \label{lem:bisections reg open}
Suppose that $\G$ is locally compact \'etale groupoid. Suppose that $\G$ is effective. If $B
\subseteq \G$ is an open bisection and $K \subseteq B$ is relatively closed in $B$, then $K^\circ$
is regular open. In particular, if $K \subseteq B$ is compact, then $K^\circ$ is regular open.
\end{lemma}
\begin{proof}
For $x \in r(B)$, let $\alpha_x$ be the unique element of $B$ with $r(\alpha_x) = x$, and define
$T_B : r(B) \G \to s(B) \G$ by $T_B(\gamma) = \alpha_{r(\gamma)}^{-1} \gamma$. Since $B$ is a
bisection, $T_B$ is a homeomorphism between the open subsets $r(B) \G$ and $s(B) \G$. Since
$T_B(V) = s(V)$ for all $V \subseteq B$, we have $T_B(K^\circ) = s(K^\circ) = s(K)^\circ$. Since
$T_B$ is a homeomorphism and $\overline{K} \subseteq r(B)\G = \operatorname{dom}(T_B)$, we also
have
\[
T_B\big(\big(\overline{K}\big)^\circ\big)
    = T_B\big(\overline{K}\big)^\circ
    = \big(\overline{T_B(K)}\big)^\circ
    = \big(\overline{s(K)}\big)^\circ.
\]
So $K^\circ$ is regular open if and only if $s(K)^\circ$ is regular open. Since $T_B$ is a
homeomorphism, $s(K)$ is relatively closed in $s(B)$. Since $s(B)$ is open, it follows that $s(K)$
is relatively closed in $\go$. So $s(K)^\circ$ is regular open by Corollary~\ref{cor:closed unit
set}. The final statement follows because $B$ is Hausdorff and so compact subsets of $B$ are
relatively closed in $B$.
\end{proof}

\medskip

\noindent \textbf{Ample groupoids and Steinberg algebras}

\medskip

An \emph{ample groupoid}  is an \'etale groupoid that has a basis of {\em compact} open
bisections. This is the class of groupoids for which there is an associated \emph{Steinberg
algebra} \cite{St}. Throughout this paper, $\KK$ denotes a field. The Steinberg $\KK$-algebra
associated to an ample groupoid $\G$ is
\[A_\KK(\G):= \lsp\{1_B : B \text{ is a compact open bisection }\}\]
where $1_B$ denotes the characteristic function of $B$, and where addition and scalar
multiplication are defined pointwise and multiplication is given by convolution:
\[ f*g(\gamma)=\sum_{r(\eta)=r(\gamma)} f(\eta)g(\eta\inv\gamma).\]

In the Hausdorff setting, the support of every function in $A_\KK(\G)$ is compact open. This is
not true in the non-Hausdorff setting and pinpointing exactly where a function is nonzero can be
tricky.  For $f \in A_\KK(\G)$ we can write  \[f= \sum_{D \in F}a_D1_D\] where $F$ is a finite
collection of compact open bisections and for each $D \in F, a_D \in \KK$. We can ``disjointify''
\footnote{We point out that there is an error in the description of ``disjointification'' in
\cite{CFST}.}
\footnote{Caution:  in a non-Hausdorff space, the intersection of compact sets might not be
compact.  Also, compact sets need not be closed.}
the collection $F$ and write $f$ as a sum of characteristic functions, each of which is nonzero on
a set of the form
\begin{equation}\label{eq:support}\left(\bigcap_{B \in F_1} B \right) \setminus \left(\bigcup_{D \in F_2} D\right)\end{equation}
where $F_1$ and $F_2$ are finite collections of compact open bisections.
Thus, in general, we can say that  $f$ is nonzero precisely on a finite union of
pairwise disjoint sets the form given
in \eqref{eq:support}.  Further, if $k \in f(\G)\setminus \{0\}$, then
$f^{-1}(k)$ is also equal to a finite union of sets of the form given in \eqref{eq:support}.

\medskip

\noindent \textbf{The ``support'' of a function}

\medskip

If $X$ is a topological space, $\KK$ is a (possibly topological) field, and $f:X\to \KK$ is a function, we write
\begin{equation}\label{eq:suppdef}
\supp(f) = \{x \in X: f(x) \neq 0\} \text{ and}
\end{equation}
\begin{equation}\label{eq:closedsuppdef}
\cls(f) = \overline{\{x \in X: f(x) \neq 0\}}.
\end{equation}
and call these the {\em support} and {\em closed support} of $f$, respectively.

\section{Simplicity of Steinberg algebras associated to ample groupoids}
\label{sec:st}

In the following lemma, we introduce the conditions we later connect to simplicity.

\begin{lemma}
\label{lem:key}
Let $\G$ be an ample groupoid such that $\go$ is Hausdorff.
Consider the following statements:
\begin{enumerate}
 \item \label{lemit1:key} Every compact open subset of $\G$ is regular open.
 \item \label{lemit2:key} For every $f \in A_\KK(\G)$ and $k \in f(\G) \setminus \{0\}$,
the set $f^{-1}(k)$ has nonempty interior.
\item \label{lemit3:key}  For every nonzero $f \in A_\KK(\G)$ the set
$\supp(f)$
has nonempty interior.
\end{enumerate}
Then \eqref{lemit1:key} $\implies$ \eqref{lemit2:key} $\implies$ \eqref{lemit3:key}.
Firther if $\KK$ has characteristic 0, then  \eqref{lemit1:key} and \eqref{lemit2:key} are equivalent.
\end{lemma}

\begin{proof}
Suppose item~\eqref{lemit1:key} holds.  Fix $f \in A_\KK(\G)$ and $k \in f(\G) \setminus \{0\}$.
As described in~\eqref{eq:support},  there are finite collections $F_1$ and $F_2$ of
compact open bisections such that the nonempty difference
\[
 \big(\bigcap_{B \in F_1} B \big) \setminus \big(\bigcup_{D \in F_2} D\big) \subseteq f^{-1}(k).
\]
So it suffices to show that, writing
\[B:=\bigcap_{B \in F_1} B \quad \text{and} \quad D:=\bigcup_{D \in F_2} D,\]
the set $B \setminus D$ has nonempty interior.

Since $D$ is compact open, it is regular open by assumption. Since the intersection of
regular open sets is again regular open, $B$ is regular open too. Thus
we can apply Lemma~\ref{lem:regopendif} to see that $B \setminus D$ has nonempty interior giving
\eqref{lemit1:key}$\;\implies\;$\eqref{lemit2:key}.  The implication \eqref{lemit2:key} implies \eqref{lemit3:key} is immediate.

Now suppose $\KK$ has characteristic 0.  To see that \eqref{lemit2:key} implies \eqref{lemit1:key}, we show the contrapositive.  Suppose there exists a compact open
set $V$ that is not regular open.  Because $\G$ is ample, we can write \[V = \bigcup_{B \in F} B\]
where $F$ is a finite collection of compact open bisections.  Since $V$ is not regular open, we can find a compact
open bisection $D$ such that $D \subseteq \overline{V}$ but $D \nsubseteq V$.  Thus $D \setminus V$ is nonempty and has empty interior.  Consider the function \[f = 1_D - \sum_{B \in F} 1_B \in A_\KK(\G).\]
Since $\KK$ has characteristic 0,  $f^{-1}(1) = D \setminus V$ and hence item \eqref{lemit2:key}  is false
\end{proof}

We now record a consequence of $\supp(f)$ having empty interior.

\begin{lemma}\label{lem:SupportElementBoundary}
	Let $\G$ be an ample groupoid such that $\go$ is Hausdorff, and suppose we have $f\in A_\KK(\G)$ such that $\supp(f)$ has empty interior.
	Then for all $\gamma\in\supp(f)$, there exists some compact open bisection $D$ such that $\gamma\in \overline{D}\setminus D$.
\end{lemma}
\begin{proof}
	Suppose that $f\in A_\KK(\G)$ such that $\supp(f)$ has empty interior.
	If $\supp(f) = \emptyset$ we are done, so suppose $\gamma\in \supp(f)$ and write $f(\gamma) = k \neq 0$.
	Then as in the proof of Lemma~\ref{lem:key}, there are finite collections $F_1$ and $F_2$ of
	compact open bisections such that $\gamma\in \big(\bigcap_{B \in F_1} B \big) \setminus \big(\bigcup_{D \in F_2} D\big)\subseteq f^{-1}(k)$.
	We claim that $\gamma$ is in the closure of some element of $F_2$.
	If not, then there must be some open set $O$ around $\gamma$ such that $O\cap D = \emptyset$ for all $D\in F_2$, which would imply that $
	O\cap\big(\bigcap_{B \in F_1} B \big)$ is an open set inside $f^{-1}(k)$, a contradiction.
	Hence there exists $D\in F_2$ such that $\gamma\in \overline D$, and since $\gamma$ is not in $D$ by assumption we are done.
\end{proof}

In order to determine when the groupoid condition
of Lemma~\ref{lem:key}~\eqref{lemit1:key}
holds, it is not enough to show that every
compact open bisection is regular open.
In fact, when $\G$ is effective, every compact open bisection
is regular open by Lemma~\ref{lem:bisections reg open}
yet Lemma~\ref{lem:key}~\eqref{lemit1:key} can still fail to hold.
See for example the Grigorchuk group of Section~\ref{sec:grig} which demonstrates that
 (3) is strictly weaker than (1) and (2).

\subsection{Singular elements}
In this subsection, we identify an important ideal of $A_\KK(\G)$ which does not appear in the non-Hausdorff case.
Recall that $A_\KK(\G)$ consists of all functions of the form
\[
f = \sum_{D\in F}a_D1_D
\]
where $F$ is a finite set of compact open bisections and $a_D\in \KK$.
For any subset $J\subseteq F$, define
\[
M_J:= \left(\bigcap_{B\in J}B\right) \setminus \left( \bigcup_{D\notin J}D \right) = \left(\bigcap_{B\in J}B\right) \cap \left(\bigcap_{D\notin J} (\G\setminus D)\right)
\]
so that the union of the elements of $F$ can be written as the {\em disjoint} union of the $M_J$.
Moreover, $f$ is constant on $M_J$, so we can rewrite $f$ as
\begin{equation}\label{eq:MJ}
f = \sum_{\emptyset \neq J\subseteq F}c_J1_{M_J}
\end{equation}
where the sum ranges over all nonempty subsets $J\subseteq F$, and $c_J\in \KK$. One notices that
\[
O_J : = \bigcap_{B\in J}B
\]
is an open bisection (which is not necessarily compact) while
\[
C_J:=\bigcap_{D\notin J}\G\setminus D
\]
is a closed set.
Given that $M_J = O_J\cap C_J$, we see that $M_J$ is a relatively closed subset of the open bisection $O_J$.
This leads us to the following definition.

\begin{definition}\label{def:singular}
	We will say that $f\in A_\KK(\G)$ is {\em singular} if $f$ can be written as a linear combination of the form
	\begin{equation}\label{eq:singularform}
	f = \sum_{i = 1}^nc_i1_{M_i}
	\end{equation}
	where the collection of $M_i$'s is pairwaise disjoint, each $M_i$ is a relatively closed subset of some open bisection, and each $M_i$ has empty interior.
	We let
	\begin{equation}\label{eq:singulardef}
	\s_\KK(\G) = \{f\in A_\KK(\G): f\text{ is singular}\}.
	\end{equation}
\end{definition}

We now prove a pair of topological lemmas which will help us characterize singular elements.
\begin{lemma}\label{lem:relclosed}
	Let $M$ be a relatively closed subset of an open set $O$ in some topological space, and suppose $M$ has empty interior. Then $\overline M\cap O = M$ and $\overline{M}$ has empty interior.
\end{lemma}
\begin{proof}
	Write $M = O\cap C$ for some closed set $C$.
	Then $\overline M \subseteq C$ and $\overline M \subseteq \overline O$.
	Then
	\[
	M = M\cap O \subseteq\overline M \cap  O \subseteq C\cap O = M \hspace{1cm} \Rightarrow \overline M\cap O = M.
	\]
	Suppose that $U\subseteq \overline M$ is an open set.
	Then
	\[
	U\cap O \subseteq \overline M \cap O = M
	\]
	and since $M$ has empty interior, we have $U\cap O = \emptyset$.
	We have $U\subseteq \overline M \subseteq \overline O$, so
	\[
	U \subseteq \overline O \setminus O \subseteq \partial O
	\]
	and since the boundary of any open set in any topological space has empty interior, we have $U = \emptyset$ and hence $\overline{M}$ has empty interior.
\end{proof}

\begin{lemma}\label{lem:unionemptyinterior}
	Suppose that $\{O_j\}_{j = 1}^{n}$ is a finite collection of open sets in a topological space, and suppose that for all $j$, $M_j$ is a relatively closed subset of $O_j$ with empty interior. Then $\cup_{j=1}^n M_j$ has empty interior.
\end{lemma}
\begin{proof}
	By Lemma \ref{lem:relclosed} we can assume that for each $j$, $M_j = O_j \cap C_j$ for a closed set $C_j$ with empty interior.  By way of contradiction
	suppose that $U\subseteq \cup_{j=1}^n M_j$ is a nonempty open set.
	Then
	\[
	U \setminus M_1 = U \setminus (O_1\cap C_1) = (U\setminus O_1) \cup (U\setminus C_1).
	\]
	Since $C_1$ has empty interior by assumption, $U$ is not contained in $C_1$, so $U\setminus C_1$ is a nonempty open set contained in $U \setminus M_1$.
	Since
	\[
	U\setminus M_1 = U \cap M_1^c \subseteq \left(\bigcup_{j=1}^n M_j\right) \cap M_1^c\subseteq \bigcup_{j = 2}^nM_j
	\]
	we must have that $\cup_{j=2}^n M_j$ has nonempty interior.
	Continuing inductively, we end up concluding that $M_n$ has nonempty interior, a contradiction.
	Hence, $\cup_{j=1}^n M_j$ has empty interior.
\end{proof}
We now have the following characterization of singular elements.
\begin{prop}\label{prp:singular empty int}
	A function $f\in A_\KK(\G)$ is singular if and only if $\supp(f)$ has empty interior.
\end{prop}
\begin{proof}
	If $f\in A_\KK(\G)$ is any element for which $\supp(f)$ has empty interior, then upon writing $f$ in the form \eqref{eq:MJ} and discarding the terms corresponding to vanishing $c_J$, the remaining $M_J$ must have empty interior, since they are contained in
the support of $f$. Hence $f$ is singular.
	
	For the other implication, suppose that $f$ is singular, written as in \eqref{eq:singularform}.
	Then $\supp(f)\subseteq \cup_{i=1}^n M_n$, which has empty interior by Lemma~\ref{lem:unionemptyinterior}.
\end{proof}
\begin{prop}\label{prop:singularideal}
	Let $\G$ be an ample groupoid  with Hausdorff unit space. Then the set $\s_\KK(\G)$ of singular elements is an ideal of $A_\KK(\G)$.
\end{prop}
\begin{proof}
	It is clear that $\s_\KK(\G)$ is closed under addition and scalar multiplication.
	We need to show that $fg$ and $gf$ are in $\s_\KK(\G)$ for all $f\in \s_\KK(\G)$ and $g\in A_\KK(\G)$, and by linearity this will be accomplished if we can show that $1_B1_M = 1_{BM}$ and $1_M1_B = 1_{MB}$ are in $\s_\KK(\G)$, where $B$ is a compact open
bisection and $M$ is a relatively closed subset of some open bisection such that $M$ has empty interior.
	
	So let $B$ be a compact open bisection and let $M$ be a relatively closed subset of some open bisection $O$.
	We claim that $MB$ is a relatively closed subset of the open bisection $OB$ and that the interior of $MB$ is empty.
	
	To see that $MB$ is closed in $OB$, let $\{\gamma_i\}$ be a net in $MB$ converging to some $\gamma \in OB$.
	Write $\gamma_i = m_ib_i$ and $\gamma = xb$ with $m_i\in M$, $b_i, b\in B$ and $x\in O$.
	Then
	\[
	s(b_i) = s(m_ib_i) = s(\gamma_i) \to s(\gamma) = s(b),
	\]
	so $b_i\to b$ due to the fact that $s$ is a homeomorphism from $B$ to $s(B)$.
	Consequently
	\[
	m_i = m_ib_ib_i^{-1} = \gamma_ib_i^{-1} \to \gamma b^{-1} = x.
	\]
	We conclude that $x$ lies in the closure of $M$ relative to $O$, and so $x\in M$, because $M$ is closed in $O$.
	Consequently,
	\[
	\gamma =xb \in MB,
	\]
	thus proving that $MB$ is closed in $OB$.
	
	To see that the interior of $MB$ is empty, assume otherwise, so that there is a nonempty open bisection $A\subseteq MB$.
	We then have
	\[
	AB^{-1} \subseteq MBB^{-1} \subseteq M.
	\]
	Clearly $AB^{-1}$ is an open bisection, so it must be empty, because $M$ has empty interior.
	Observing that $s(A)\subseteq s(B)$, we then have that
	\[
	A = As(A) \subseteq As(B)= AB^{-1}B = \emptyset,
	\]
	contradicting our choice of $A$.
	This shows that $1_{MB} = 1_M1_B \in \s_\KK(\G)$.
	A similar argument shows that $1_{BM} = 1_B1_M \in \s_\KK(\G)$, and we are done.
	
\end{proof}
We will return to simplicity in Section~\ref{sec:steinbergsimplicity}, but for now we note the following consequence of the above.

\begin{cor}\label{cor:simplesupport}
	Let $\G$ be a second-countable \'etale groupoid and with Hausdorff unit space.
	If $A_\KK(\G)$ is simple, then $\supp(f)$ has nonempty interior for all nonzero $f\in A_\KK(\G)$.
\end{cor}

\subsection{A Uniqueness theorem for Steinberg algebras}

\medskip

The proof of  our uniqueness theorem is a non-Hausdorff version of
\cite[Lemma~3.3]{CEP}\footnote{See the sentence before \cite[Example~3.5]{CEP}.} which is our Lemma~\ref{lem:CEP2} below
(and is an algebraic analogue of \cite[Lemma~3.3(b)]{BNRSW}.)
The non-Hausdorff proof is almost exactly the same as the Hausdorff version; we
include the details.   First we establish a non-Hausdorff version of \cite[Lemma~3.3(a)]{BNRSW}.

\begin{lemma}
\label{lem:CEP1}
\label{lem:denseuu}Let $\G$ be a second-countable, ample groupoid such that $\go$ is Hausdorff.
Then \[X := \{u \in \go : \G^u_u \subseteq \intiso \}\]
is dense in $\go$.
\end{lemma}
\begin{proof}
Fix a compact open bisection $B$ of $\G$. Let
$B':=B\cap (\Iso(\G) \setminus \intiso)$.
Since $\Iso(\G)$ is closed in $\G$,  $B'$ is closed in $B$ with respect to the relative topology,
$r(B')$ is closed in $\go$.
We claim that $r(B')$ has empty interior.
Suppose that $V$ is an open subset contained in $r(B')$. We show that
$V$ is empty.  Since $B$ is open, \[VB = r^{-1}(V) \cap B \] is open too.  Further, since
$B$ is a bisection, $r$ is injective on $B$ and hence $VB \subseteq B'$.
Thus $VB$ is an open subset of
$\Iso(\G) \setminus \intiso$.  But
$\Iso(\G) \setminus \intiso$ has empty interior.
Therefore $V=\emptyset$ proving the claim.

Let
\[C:=\{u\in  \go:\G_u^u \nsubseteq \intiso\}.\]
Since  $\G$  has a basis of compact open
bisections,
\[C=\{r(B\cap (\Iso(\G)\setminus\intiso)) :B
\text{ is a compact open bisection }\}.\]
So $\G$ being second countable implies $C$ is a countable union of nowhere-dense sets.
By applying the Baire Category Theorem in the locally compact Hausdorff space $\go$, we
see that $C$ is nowhere dense. Hence
its complement $\{u\in  \go:\G_u^u \subseteq \intiso\}$ is dense in
$\go$.
\end{proof}

\begin{lemma}
 \label{lem:CEP2}
Let $\G$ be a second-countable, ample groupoid such that $\go$ is Hausdorff.
 Suppose $u \in \go $ is such that
 $\G^u_u \subseteq \intiso $ and take $f \in A_\KK(\G)$ such that
 there exists $\gamma _u \in \G^u_u$ with
 $f(\gamma _u) \neq 0$.  Then
 there exists a compact open set $L \subseteq \go $ such that $u \in L$ and
 \begin{equation}
  \label{eq:algsupport}\emptyset \neq \{\gamma \in \G : 1_Lf1_L(\gamma) \neq 0\} \subseteq \intiso.
 \end{equation}

\end{lemma}

\begin{proof}
 Write $f= \sum_{D \in F}a_D1_D$ where $D$ is a collection of compact open bisections.  For each $D \in F$,
 choose a compact open neighbourhood $V_D \subseteq \go$ of $u$ as follows:
 \begin{itemize}
  \item If $u = r(\gamma) = s(\gamma)$ for some $\gamma \in D$, then
  $\gamma \in \intiso$ by assumption.  Let
  $V_D$ be a compact open subset in $\go$ containing $u$ such that $V_D$
  is contained in the open set \[r(D \cap \intiso) = s(D \cap \intiso).\]
Then  $V_D D V_D \subseteq D \cap \intiso$.
\item If there exists $\gamma \in D$ such that $r(\gamma) = u$ and $s(\gamma) \neq u$ or $s(\gamma) = u$ and
$r(\gamma) \neq u$, we chose $V_D$ as follows.  Because $\go$ is Hausdorff, we can find a compact open subset
$D' \subseteq D$ containing $\gamma$ such that $r(D') \cap s(D') = \emptyset$.
Take $V_D = r(D')$ (or $V_D= s(D')$), so that $u \in V_D$ and  $V_D DV_D = \emptyset$.
\item If $u \notin r(D)$ and $u \notin s(D)$, use that $\go$ is Hausdorff to choose a compact open neighbourhood
$V_D$ of $u$ such that $V_D DV_D=\emptyset$.
 \end{itemize}
Let \[L:= \bigcap_{D \in F} V_D.\]  Then $L$ is a compact open set that contains $u$.
 Further the hypotheses about $u$ and $f$ imply
$1_Lf1_L(\gamma_u) \neq 0$ and by construction
\eqref{eq:algsupport} holds.
\end{proof}

\begin{thm}
\label{thm:cku}
Let $\G$ be a second-countable ample groupoid such that $\go$ is Hausdorff,
$\G$ is effective, and $\supp(f)$ has nonempty interior for all nonzero $f\in A_\KK(\G)$.
 Suppose $A$ is a $\KK$-algebra and $\pi:A_\KK(\G) \to A$ is a ring homomorphism with nonzero kernel.
 Then there exists a compact open $L \subseteq \go$ such that $1_L \in \operatorname{ker}(\pi)$.
\end{thm}

\begin{proof}
 Fix nonzero $f \in \operatorname{ker}(\pi)$.
 Since $\supp(f)$ has nonempty interior, by writing $f$ in the form \eqref{eq:MJ} and applying  Lemma~\ref{lem:unionemptyinterior},
 find a compact open bisection
$B$ such that $f$ is a nonzero constant on $B$.
 Consider the function $g:= f*1_{B^{-1}} \in \ker(\pi)$.

 We claim that $g(u) \neq 0$ for all $u \in r(B)$.
To see this, fix $u = bb^{-1} \in r(B)$.  Then
 \[g(u) = f*1_{B^{-1}}(u) = \sum_{r(\beta) = u} f(\beta) 1_{B^{-1}}(\beta^{-1}) = f(b) \neq 0.\]

Because $\G$ is second countable and effective, $\G$ is topologically
principal (for example, see \cite[Proposition~3.6]{Ren2}).
Thus, there exists $u \in r(B) \subseteq \go$ such that $\G^u_u = \{u\}$ and $g(u) \neq 0$.
Now we apply  Lemma~\ref{lem:CEP2} to find
 a compact open set $M \subseteq \go$ such that for
$h:=1_Mg1_M$, we have \eqref{eq:algsupport} holds for $h$.
Notice $h \in \ker(\pi)$.

Since $\G$ is effective,  $\intiso = \go$ and hence
\[\{\gamma \in \G: h(\gamma) \neq 0\} \subseteq \go.\]
By assumption, $\supp(h)$ has nonempty interior so again,
we find a compact open set
$L \subseteq \go $ such that $h$ is constant on $L$.
Now
\[
 1_L = (h(L))^{-1}1_L h \in \ker(\pi).\]
\end{proof}

\begin{cor}
\label{cor:cku}
Let $\G$ be a second-countable ample groupoid such that $\go$ is Hausdorff,
$\G$ is effective, and $\supp(f)$ has nonempty interior for all nonzero $f\in A_\KK(\G)$.
 Suppose $I$ is a nonzero ideal in $A_\KK(\G)$.
 Then there exists a nonempty compact open $L \subseteq \go$ such that $1_L \in I$.
\end{cor}


\subsection{Simplicity of Steinberg algebras}
\label{sec:steinbergsimplicity}
\medskip

We are now in a position to generalize the following theorem:

\begin{thm}{\cite[Theorem 3.5]{St2}}
\label{thm:SteinbergSimple}
Let $\G$ be an ample groupoid such that $\go$ is Hausdorff. If $A_\KK(\G)$ is simple,
then $\G$ is effective and minimal. The converse holds if $\G$ is Hausdorff.
\end{thm}

\begin{thm}
\label{thm:simple}
 Let $\G$ be a second-countable, ample groupoid such that $\go$ is Hausdorff.
 Then $A_\KK(\G)$ is simple if and only if the following three conditions are satisfied:
 \begin{enumerate}
  \item\label{it1:thm} $\G$ is minimal,
  \item\label{it2:thm} $\G$ is effective, and
  \item\label{it3:thm} for every nonzero $f\in A_\KK(\G)$, $\supp(f)$ has nonempty interior.
 \end{enumerate}

\end{thm}

\begin{proof}
 Suppose items \eqref{it1:thm}--\eqref{it3:thm} are satisfied.  Let $I$ be a nonzero ideal in $A_\KK(\G)$.
 By Corollary~\ref{cor:cku}, there exists a compact open set $L \subseteq \go$ such that $1_L \in I$.
We show that for each $x \in \go$, there exists a compact open $L_x$ containing $x$ such that $1_{L_x} \in I$.
Using item \eqref{it1:thm}, there exists $\gamma_x \in \G$ such that $r(\gamma_x)=x$ and $s(\gamma_x) \in L$.
Let $B_x$ be a compact open bisection containing $\gamma_x$.
Then the function
\[1_{B_x}1_{L}1_{B_x^{-1}} = 1_{B_xLB_x^{-1}} \in I.\]
 is nonzero at $x$.
 So the compact open subset $L_x := B_xLB_x^{-1} \subseteq \go$ suffices.
 Thus, we have $1_M \in I$ for any compact open subset $M$ of $\go$ and hence $I = A_\KK(\G)$.

 Now suppose $A_\KK(\G)$ is simple.
 Items~\eqref{it1:thm} and \eqref{it2:thm}
 hold by Theorem~\ref{thm:SteinbergSimple}, and  Corollary~\ref{cor:simplesupport} implies \eqref{it3:thm}.
\end{proof}

\begin{remark}
Theorem~\ref{thm:simple} is related to \cite[Proposition~4.1]{Nek16}, which was brought to our
attention after we posted this paper to the \emph{arXiv}. Specifically, suppose that
$\G$ is second-countable, ample, effective, and minimal. Then $\G$ is also topologically principal, and
so \cite[Proposition~4.1]{Nek16} shows that the collection of elements $f \in A_\KK(\G)$ that vanish
on $\mathcal{T} := \big\{\gamma \in \G : \G^{s(\gamma)}_{s(\gamma)} = \{s(\gamma)\}\big\}$ form an ideal
$I$ such that $A_\KK(\G)/I$ is simple. Since $\G \setminus \mathcal{T}$ has empty interior,
Proposition~\ref{prp:singular empty int} shows that $I$ is contained in the set $\s_\KK(\G)$ of singular
elements, which is also an ideal by Proposition~\ref{prop:singularideal}. So $A_\KK(\G)/\s_\KK(\G)$ is a
quotient of $A_\KK(\G)/I$. Since $\s_\KK(\G)$ is not all of $A_\KK(\G)$---it has trivial intersection with
$C_0(\G^{(0)})$ by Proposition~\ref{prp:singular empty int}---we see that $A_\KK(\G)/\s_\KK(\G)$ is nonzero.
Since $A_\KK(\G)/I$ is simple, it follows that $\s_\KK(\G) = I$. So the ``if" direction of
Theorem~\ref{thm:simple} can be recovered from \cite[Proposition~4.1]{Nek16}.
\end{remark}

By applying Lemma~\ref{lem:key}, we get the following corollary.

\begin{cor}
\label{cor:Stsimple}
 Let $\G$ be a second-countable, ample groupoid such that $\go$ is Hausdorff.
Suppose following three conditions are satisfied:
 \begin{enumerate}
  \item\label{it1:thmcs} $\G$ is minimal,
  \item\label{it2:thmcs} $\G$ is effective, and
  \item\label{it3:thmcs}every compact open subset of $\G$ is regular open.
 \end{enumerate}
  Then $A_\KK(\G)$ is simple.
\end{cor}

\section{Groupoid \texorpdfstring{$\cs$}{C*}-algebras}
\label{sec:an}

\subsection{Groupoid \texorpdfstring{$\cs$}{C*}-algebra preliminaries}

\medskip

In this section, we no longer restrict our attention to ample groupoids, instead we consider
arbitrary (second-countable, locally-compact) \'etale groupoids (with Hausdorff unit space).  Here we mainly deal with the `open' support,
(unconventionally) defined earlier in \eqref{eq:suppdef} as
\[\supp(f) = {\{x \in X: f(x) \neq 0\}}.\]
We  denote the set of continuous
functions with compact support by
\[C_c(X) :=\{ f:X \to \CC : f \text{ is continuous and $\exists$ compact $K$ such
that $f(x) = 0 \ \forall \  x \notin K$}\}.\]
We write $\Cc(\G)$ for Connes' algebra of
functions $f : \G \to \CC$ linearly spanned by the spaces $C_c(U)$ for open
bisections $U$ contained in $\G$.   We view a function in $C_c(U)$ as a function on $\G$ by
defining it to be 0 outside of $U$. In some papers and texts, this algebra $\Cc(\G)$ is simply
denoted $C_c(\G)$, but we avoid this since its elements are in general not continuous.

For $u \in \go$, we write $L_u$ for the regular representation $L_u : \Cc(\G) \to
\Bb(\ell^2(\G_u))$ satisfying
\[L_u(f)\delta_\gamma = \sum_{\alpha \in \G_{r(\gamma)}}
f(\alpha)\delta_{\alpha\gamma} \text{ for } f \in \Cc(\G).\]
By definition, $\cs_r(\G)$ is the
completion of the image of $\Cc(\G)$ under $\bigoplus_{u \in \go} L_u$.

\label{page_functions}
Recall that if $\G$ is an \'etale groupoid and $a \in \cs_r(\G)$, then we can define a
function $j(a) : \G \to \CC$ by $j(a)(\gamma) = \big(L_{s(\gamma)}(a) \delta_{s(\gamma)}
\mid \delta_\gamma\big)$. It is routine to check that $j(f) = f$ for $f \in \Cc(\G)$.
Since $\bigoplus_u L_u$ is faithful, the map $a \mapsto j(a)$
is injective. Write $B(\G)$ for the vector space of all bounded functions $f : \G \to \CC$,
and regard $B(\G)$ as a normed vector space under $\|\cdot\|_\infty$. For $a \in \cs_r(\G)$
and $\gamma \in \G$, we have
\[
|j(a)(\gamma)| = \big|\big(L_{s(\gamma)}(a)\delta_{s(\gamma)} \mid \delta_\gamma\big)\big| \le \|L_{s(\gamma)}(a)\| \le \|a\|,
\]
so $j$ is an injective norm-decreasing linear map from $\cs_r(\G)$ to $B(\G)$.

\subsection{Singular elements}

As in the Steinberg algebra setting, the presence of any singular elements gives rise to a nontrivial ideal.  Here we call $a \in \cs_r(\G)$ \emph{singular} if $\supp(j(a))$ has empty interior.     The following lemma gives some useful insight.

\begin{lemma}
\label{lem:key'} Let $\G$ be a locally compact, \'etale groupoid such that $\go$ is Hausdorff.
Consider the following conditions:
\begin{enumerate}
\item \label{lemit2:key'} For every $f \in \Cc(\G)$, every $z \in f(\G) \setminus\{0\}$
    and every $\varepsilon > 0$, the set \[\{\gamma \in \G : |f(\gamma) - z| <
    \varepsilon\}\] has nonempty interior.
\item \label{lemit3:key'} For every nonzero $a \in \cs_r(\G)$,  $\supp(j(a))$
    nonempty interior.
\end{enumerate}
Then (\ref{lemit2:key'}) $\implies$  (\ref{lemit3:key'}).
\end{lemma}

\begin{proof}
Choose a nonzero $a \in \cs_r(\G)$,
and fix $\gamma$ with $j(a)(\gamma) \not= 0$. Fix $f \in \Cc(\G)$ with $\|a - f\|_r <
|j(a)(\gamma)|/3$. Since $j$ is norm-decreasing, we see that $\|j(a) - f\|_\infty <
|j(a)(\gamma)|/3$ and so $|f(\gamma)| > 2|j(a)(\gamma)|/3$. By hypothesis,
\begin{equation}
\label{eq:ahh}U := \{\eta
\in \G : |f(\eta) - f(\gamma)| < |j(a)(\gamma)|/3\}
\end{equation} has nonempty interior. Since
$\|j(a) - f\|_\infty < |j(a)(\gamma)|/3$, we see that for $\eta \in U$ we have
\[
|j(a)(\eta)| > |f(\eta)| - |j(a)(\gamma)|/3 > (|f(\gamma)| - |j(a)(\gamma)|/3) - |j(a)(\gamma)|/3
    > 0.
\]
So $U \subseteq \supp(j(a))$, and since $U$ has nonempty interior, so does
$\supp(j(a))$.
\end{proof}

In what follows, we denote the collection of units with trivial isotropy as $S$.  That is
\[S:=\{u \in \go : \G^u_u = \{u\}\}.\]
Recall that $S$ and $\go\setminus S$ are {\em saturated}; that is $[S] = S$ and $[\go\setminus S] = \go\setminus S$.

\begin{lemma}
\label{lem:singiso}
 Let $\G$ be a locally compact, \'etale groupoid such that $\go$ is Hausdorff.  Suppose that
 $a \in \cs_r(\G)$ is a singular element.
 \begin{enumerate}
 \item\label{item1:lemahh} For every $\gamma \in \supp(j(a))$, there exists $f \in \Cc(\G)$ such that $\gamma \in \supp(f)$ and $f$ is discontinuous at $\gamma$.
 \item\label{item2:lemahh} We have $s(\supp(j(a))) \subseteq \go\setminus S$.
 \end{enumerate}
\end{lemma}

\begin{proof}
Suppose $a \in \cs_r(\G)$ is singular, that is  $\supp (j(a))$ has empty interior.
For \eqref{item1:lemahh}, as in the proof of Lemma~\ref{lem:key'}, find $f \in \Cc(\G)$
with  $\|a - f\|_r < |j(a)(\gamma)|/3$.   Then the set $U$ as defined in \eqref{eq:ahh},  must have empty interior, because otherwise the rest of the proof of Lemma~\ref{lem:key'} would imply that $\supp(j(a))$ has nonempty interior.  Using a decresasing
neighbourhood base, we can find a sequence $\{\gamma_n\}$ such that $\gamma_n$ converges to $\gamma$ but $\gamma_n \notin U$.
That is
\begin{equation}\label{eq:noconverge1}|f(\gamma_n) - f(\gamma)| \geq \epsilon\end{equation}
where $\epsilon= |j(a)(\gamma)|/3$.  Item \eqref{item1:lemahh} follows.

For item~\eqref{item2:lemahh}, write the $f$ from item~\eqref{item1:lemahh} as $f = \sum_{D \in F}f_D$ where $F$ is a finite collection of open bisections and each $f_D \in C_c(D)$.
We know from \eqref{eq:noconverge1} that $f(\gamma_n)$ does not converge to $f(\gamma)$, so there must exist $D_1\in F$ such that
\begin{equation}\label{eq:noconverge2}
f_{D_1}(\gamma_n)\not\to f_{D_1}(\gamma)
\end{equation}
and so $f_{D_1}$ is not continuous at $\gamma$.
If $\gamma\in D_1$ then $\gamma_n$ is in $D_1$ eventually, and since $f_{D_1}$ restricted to $D_1$ is continuous this contradicts \eqref{eq:noconverge2}.
Hence $\gamma\notin D_1$ which implies $f_{D_1}(\gamma) = 0$.
By \eqref{eq:noconverge2}, there exists a subsequence $\{\gamma_{n_k}\}$ such that $\{|f_{D_1}(\gamma_{n_k})|\}$ is bounded away from zero, in particular they are nonzero.
Hence $\gamma_{n_k}\in D_1$ for all $k$, and since $f_{D_1}$ is supported on a compact subset of $D_1$ which must necessarily contain all the $\gamma_{n_k}$, we can pass to a convergent subsequence $\gamma_{n_{k_\ell}}\to \gamma_1\in D_1$.
Because $f_{D_1}$ is continuous on $D_1$ and $\{|f_{D_1}(\gamma_{n_k})|\}$ is bounded away from zero, we have $f_{D_1}(\gamma_{n_{k_\ell}})\to f_{D_1}(\gamma_1) \neq 0$.
But $f_{D_1}(\gamma) = 0 \neq f_{D_1}(\gamma_1)$, so $\gamma \neq \gamma_1$.

Since $r$ and $s$ are continuous, we have
 $s(\gamma_1)=s(\gamma)$ and $r(\gamma_1)=r(\gamma)$.  Thus $\gamma\gamma_1^{-1} \in \operatorname{Iso}(\G) \setminus \go$.  Thus $s(\gamma) \in \go \setminus S$.
\end{proof}

\begin{prop}
\label{prop:simpleconverse}
 Let $\G$ be an effective, second countable, locally compact, \'etale groupoid such that $\go$ is Hausdorff.  If $\cs_r(\G)$ has any singular elements, then $\cs_r(G)$ is not simple.
\end{prop}

\begin{proof}
Suppose  $a \in \cs_r(\G)$ is singular. Lemma~\ref{lem:singiso} gives
$s(\supp(j(a))) \subseteq~\go~\setminus S$.  Since $\G$ is second countable and effective, it is topologically principal by \cite[Proposition~3.6]{Ren2}.  Thus there exists $u \in S$.
Then  $L_u(a) = 0$ and hence the kernel of $L_u$ is a nontrivial ideal of $\cs_r(\G)$.
\end{proof}


\subsection{Uniqueness theorem for groupoid \texorpdfstring{$\cs$}{C*}-algebras}

\medskip

In this section, we prove a uniqueness theorem for the reduced $\cs$-algebra
$\cs_r(\G)$ of an \'etale groupoid $\G$:

 \begin{thm}\label{thm:reduced uniqueness}
Let $\G$ be a second-countable, locally compact,  \'etale groupoid such that $\go$ is Hausdorff, $\G$ is effective and for every nonzero $a \in \cs_r(\G)$, $\supp(j(a))$ has nonempty interior. Let $\rho:\cs_r(\G) \to B$ be a $\cs$-homomorphism that is
injective on
$C_0(\go)$. Then $\rho$ is injective.
\end{thm}
In our proof, we will invoke Theorem~3.2 of \cite{BNRSW}
which we restate for convenience.

\begin{thm}{\cite[Theorem~3.2]{BNRSW}}
\label{thmBNRSW}
 Let $A$ be a $\cs$-algebra and $M$ a $\cs$-subalgebra of $A$.  Suppose $S$ is collection
of states of $M$ such that
\begin{enumerate}
 \item\label{it1:BNRSW} every $\varphi \in S$ has a unique extension to a state $\tilde{\varphi}$ of $A$; and
\item the direct sum $\oplus_{\varphi\in S}\pi_{\tilde{\varphi}}$ of the GNS representations
associated to extensions of the elements of $S$ to $A$ is faithful on $A$.
\end{enumerate}
Let $\rho:A \to B$ be a $\cs$-homomorphism.  Then $\rho$ is injective if and only if it is injective
on $M$.
\end{thm}

We use $C_0(\go)$ for the $M$ in Theorem~\ref{thmBNRSW}.  It is a subalgebra by the following lemma.

\begin{lemma}\label{lem:inclusion}
Let $\G$ be an \'etale groupoid such that $\go$ is Hausdorff.
Then
\begin{enumerate}
\item\label{it1:inclusion}  $C_c(\go) \subseteq \Cc(\G)$
\item\label{it2:inclusion} The inclusion map $i:C_c(\go) \to \Cc(\G)$ extends to
an injective homomorphism $i:C_0(\go) \to \cs_r(\G)$.
\end{enumerate}
\end{lemma}
\begin{proof}
The unit space $\go$ is an open bisection in  $\G$, so each $f \in C_c(\go)$ belongs to
$\Cc(\G)$ by definition. The argument of \cite[Proposition~1.9]{Phillips} works verbatim in the non-Hausdorff case
to show that $i$ extends to injective $\cs$-homomorphisms.
\end{proof}

\begin{remark}
When we view elements $f \in C_c(\go)$ as  elements
of $\Cc(\G)$ via the inclusion map $i$,
$i(f)$ might not be continuous as a function on $\G$.
\end{remark}

We also need a non-ample version of Lemma~\ref{lem:CEP2} in order to establish
item~\eqref{it1:BNRSW} of Theorem~\ref{thmBNRSW}.

\begin{lemma}\label{lem:CEP2'}
Let $\G$ be a second-countable, \'etale groupoid such $\go$ is Hausdorff. Suppose $u \in
\go $ is such that $\G^u_u \subseteq \intiso $ and take $f \in \Cc(\G)$ such that there
exists $\gamma_u \in \G^u_u$ with $f(\gamma_u) \neq 0$.  Then there exists an open set $L
\subseteq \go $ such that $u \in L$ and
\[
\supp (b f b) \subseteq \intiso
\]
for every $b \in C_c(L)$; moreover, $bfb \not= 0$ whenever $b$ satisfies $b(u)
= 1$.
\end{lemma}
\begin{proof}
Write $f= \sum_{D \in F} f_D$ where $D$ is a collection of compact open bisections and
each $f_D \in C_c(D)$. For each $D \in F$, choose an open neighbourhood $V_D \subseteq
\go$ as follows:
\begin{itemize}
\item If $u = r(\gamma) = s(\gamma)$ for some $\gamma \in D$, then $\gamma \in
    \intiso$ by assumption.  Let $V_D$ be an open subset of $\go$ containing $u$ such
    that $V_D$ is contained in the open set
    \[
        r(D \cap \intiso) = s(D \cap \intiso).
    \]
    Then  $V_D D V_D \subseteq D \cap \intiso$.
\item If there exists $\gamma \in D$ such that $r(\gamma) = u$ and $s(\gamma) \neq u$
    or $s(\gamma) = u$ and $r(\gamma) \neq u$, we chose $V_D$ as follows. Because
    $\go$ is Hausdorff, we can find an open subset $D' \subseteq D$ containing
    $\gamma$ such that $r(D') \cap s(D') = \emptyset$. Take $V_D = r(D')$ (or $V_D=
    s(D')$), so that $u \in V_D$ and  $V_D DV_D = \emptyset$.
\item If $u \notin r(D)$ and $u \notin s(D)$, use that $\go$ is Hausdorff to choose a
    neighbourhood $V_D$ of $u$ such that $V_D DV_D=\emptyset$.
\end{itemize}
Let
\[
    L:= \bigcap_{D \in F} V_D.
\]
Then $L$ is an open set that contains $u$. If $b \in C_c(L)$, then
\[
    \supp (b f b) \subseteq \intiso.
\] by construction.
If $b(u) = 1$, then $(bfb)(\gamma_u) = b(u) f(\gamma_u) b(u) = f(\gamma_u) \not= 0$, and
so $bfb \not= 0$.
\end{proof}

\begin{lemma}\label{lem:unique extension}
Let $\G$ be a second-countable, effective \'etale groupoid such that for every nonzero $a \in \cs_r(\G)$,  $\supp(j(a))$ has nonempty interior. Let $u \in \go$ be a unit such that $\G^u_u = \{u\}$. Let $\epsilon_u$ be
the state of $C_0(\go)$ determined by evaluation at $u$. Then $\epsilon_u$ extends
uniquely to a state of $\cs_r(\G)$.
\end{lemma}
\begin{proof}
We follow the argument of \cite[Theorem~3.1(a)]{BNRSW}. By the argument preceding
\cite[Theorem~3.1]{Anderson}, it suffices to show that for each $a \in \cs_r(\G)$ and each
$\varepsilon > 0$, there exist $b \in C_0(\go)_+$ and $c \in C_0(\go)$ such that
$\epsilon_u(b) = \|b\| = 1$ and $\|bab - c\| < \varepsilon$.

For this, observe that by continuity it suffices to show that for each $f$ in the dense
subalgebra $\Cc(\G)$ there exists a positive $b \in C_c(\go)$ such that $bfb \in C_c(\go)$
and $\phi(b) = \|b\| = 1$. Since $\G$ is effective, we have $\intiso = \go$, and so we can
apply Lemma~\ref{lem:CEP2'} to find an open $L \subseteq \go$ such that $u \in L$ and $b
f b \in C_0(\go)$ for all $b \in C_c(L)$. Fix $b \in C_c(L)$ with $b(u) = 1$ and $0 \leq b(v)
\le 1$ for all $v \in L$. Then we have $\epsilon_u(b) = 1 = \|b\|$, so this $b$ suffices.
\end{proof}

\begin{proof}[Proof of Theorem~\ref{thm:reduced uniqueness}]
Let $M := C_0(\go) \subseteq \cs_r(\G)$. Lemma~\ref{lem:unique extension} shows that the
states $\{\varepsilon_u : \G^u_u = \{u\}\}$ have unique extension to states of $\cs_r(\G)$.
So by Theorem~\ref{thmBNRSW}, it suffices to show that the direct sum of the GNS
representations of these state extensions is faithful on $\cs_r(\G)$.

For this, first observe that for $u \in \go$, the vector state $\phi_u(a) := \big(L_u(a)
\delta_u \mid \delta_u\big)$ is an extension of $\epsilon_u$ to a state of $\cs_r(\G)$, so
if $u$ satisfies $\G^u_u = \{u\}$, then the preceding paragraph shows that this is the
unique extension. We will show that
\begin{equation}\label{eq:GNS contains regular}
\|\pi_{\phi_u}(a)\| \ge \|L_u(a)\| \text{ for every $a \in \cs_r(\G)$.}
\end{equation}
To see this, for each $\gamma \in \G_u$, let $B_\gamma$ be a compact open bisection
containing $\gamma$, and fix $f_\gamma \in C_c(B_\gamma)$ with $f_\gamma(\gamma) = 1$. So
the $f_\gamma$ belong to $\cs_r(\G)$. For each $\gamma \in \G_u$, let $h_\gamma :=
[f_\gamma]$ be the corresponding element of the GNS space $\Hh_{\phi_u}$.

We first claim that $\{h_\gamma : \gamma \in \G_u\}$ is an orthonormal set in the GNS
space $\Hh_{\phi_u}$ of $\phi_u$. To see this, fix $\gamma,\eta \in \G_u$, and calculate:
\[
\big(h_\gamma \mid h_\eta\big)
    = \phi_u(f_\eta^* f_\gamma)
    = j(f_\eta^* f_\gamma)(u)
    = (f_\eta^* f_\gamma)(u).
\]
Since $B_\eta$ and $B_\gamma$ are bisections containing $\eta$ and $\gamma$, and since
$s(\eta) = s(\gamma) = u$, we have $u \in B_\eta^{-1}B_\gamma \supseteq \supp(f_\eta^*
f_\gamma)$ if and only if $\eta = \gamma$. So if $\gamma \not= \eta$ then $\big(h_\gamma
\mid h_\eta\big) = 0$ and if $\gamma = \eta$, then
\[
\big(h_\gamma \mid h_\eta\big)
    = \sum_{\alpha\beta = u} f_\gamma^*(\alpha)f_\gamma(\beta)
    = \sum_{\alpha\beta = u} \overline{f_\gamma(\alpha^{-1})}f_\gamma(\beta)
    = |f_\gamma(\gamma)|^2 = 1.
\]
So
\[ \big(h_\gamma \mid h_\eta\big) = \delta_{\gamma,\eta},
\]
and $\{h_\gamma : \gamma \in \G_u\}$ is an orthonormal set as claimed.

We now claim that if $B$ is an open bisection in $\G$, $\gamma \in \G_u$, and $g \in
C_c(B)$, then
\[
\pi_{\phi_u}(g) h_\gamma
    = \begin{cases}
        g(\eta) h_{\eta\gamma} &\text{ if $\eta \in B \cap \G_{r(\gamma)}$}\\
        0 &\text{ if $B \cap \G_{r(\gamma)} = \emptyset$.}
    \end{cases}
\]
First suppose that $B \cap \G_{r(\gamma)} \not= \emptyset$, say $\eta \in B \cap
\G_{r(\gamma)}$. We calculate
\begin{align*}
\|\pi_{\phi_u}&(g) h_\gamma - g(\eta)h_{\eta\gamma}\|^2\\
    &= \big(\pi_{\phi_u}(g) h_\gamma - g(\eta)h_{\eta\gamma} \mid \pi_{\phi_u}(g) h_\gamma - g(\eta)h_{\eta\gamma}\big)\\
    &= \big([g*f_\gamma] - [g(\eta)f_{\eta\gamma}] \mid [g*f_\gamma] - [g(\eta)_{\eta\gamma}]\big)\\
    &= \phi_u\big((g*f_\gamma)^*(g*f_\gamma) - (g(\eta)f_{\eta\gamma})^*(g*f_\gamma)
            - (g*f_\gamma)^*(g(\eta)f_{\eta\gamma}) + |g(\eta)|^2 f_{\eta\gamma}^**f_{\eta\gamma}\big)\\
    &= \big(L_u((g*f_\gamma)^*(g*f_\gamma) \delta_u \mid \delta_u\big)
        - \big(L_u((g(\eta)f_{\eta\gamma})^*(g*f_\gamma)) \delta_u \mid \delta_u\big)\\
    &\qquad{} - \big(L_u((g*f_\gamma)^*(g(\eta)f_{\eta\gamma})) \delta_u \mid \delta_u\big)
        + |g(\eta)|^2\big(L_u(f_{\eta\gamma}^**f_{\eta\gamma}) \delta_u \mid \delta_u\big)\\
    &= \big((g*f_\gamma)^*(g*f_\gamma) - (g(\eta)f_{\eta\gamma})^*(g*f_\gamma)
            - (g*f_\gamma)^*(g(\eta)f_{\eta\gamma}) + |g(\eta)|^2 f_{\eta\gamma}^**f_{\eta\gamma}\big)(u).
\end{align*}
Since the unique element of $BB_\gamma$ with source equal to $u$ is $\eta\gamma$, we have
\begin{align*}
\big((g*f_\gamma)^*(g*f_\gamma)\big)(u)
    &= \big((g(\eta)f_{\eta\gamma})^*(g*f_\gamma)\big)(u)
    = \big((g*f_\gamma)^*(g(\eta)f_{\eta\gamma})\big)(u)\\
    &= \big(|g(\eta)|^2 f_{\eta\gamma}^**f_{\eta\gamma}\big)(u)
    = |g(\eta)|^2,
\end{align*}
so we deduce that
\[
\|\pi_{\phi_u}(g) h_\gamma - g(\eta)h_{\eta\gamma}\|^2 = 0,
\]
so $\pi_{\phi_u}(g) h_\gamma = g(\eta)h_{\eta\gamma}$. Now suppose that $B \cap
G_{r(\gamma)} = \emptyset$. Then $BB_\gamma \cap \G_u = \emptyset$, and so
\[
\|\pi_{\phi_u}(g)h_\gamma\|^2
    = \phi_u((g*f_\gamma)^*(g*f_\gamma))
    = \big(L_u((g*f_\gamma)^*(g*f_\gamma)) \delta_u \mid \delta_u\big)
    = \big((g*f_\gamma)^*(g*f_\gamma)\big)(u)
    = 0.
\]
This proves the claim.

Since $\lsp\{g : B\text{ is an open bisection and }g \in C_c(B)\}$ is dense in
$\cs_r(\G)$, it follows that $\supp\{h_\gamma : \gamma \in \G_u\}$ is invariant for
$\pi_u$. Moreover, since for every open bisection $B$ and every $g \in C_c(B)$, we have
\[
L_u(g) \delta_\gamma =
    \begin{cases}
        g(\eta)h_{\eta\gamma} &\text{ if $\eta \in B \cap G_{r(\gamma)}$}\\
        0 &\text{ if $B \cap G_{r(\gamma)} = \emptyset$,}
    \end{cases}
\]
the unitary $U : \ell^2(\G_u) \to H_{\phi_u}$ intertwines $L_u$ with the reduction of
$\pi_{\phi_u}$ to $\clsp\{h_\gamma : \gamma \in \G_u\}$. That is $\pi_{\phi_u}$ contains a
summand equivalent to $L_u$, proving~\eqref{eq:GNS contains regular}.

So to prove the theorem, it now suffices to show that $\bigoplus_{\G^u_u = \{u\}} L_u$ is
a faithful representation of $\cs_r(\G)$. So suppose that $a \in \cs_r(\G) \setminus
\{0\}$. Then $j(a)$ is nonzero, and so by hypothesis there is an open set $O \subseteq \G$
such that $j(a)(\gamma) \not= 0$ for all $\gamma \in O$.  Fix a nonempty open bisection $B$
contained in $O$. Then $s(B)$ is a nonempty open subset of $\go$ and so there exists $u \in s(B)
\cap S$. Let $\gamma \in B$ be the unique element in $B$ with $s(\gamma) = u$. Then
\[
\big(L_u(a) \delta_u \mid \delta_\gamma\big)
    = j(a)(\gamma) \not= 0
\]
because $\gamma \in O$. Hence $L_u(a) \not= 0$. Since $u \in S$, it follows that
$\bigoplus_{u \in S} L_u(a) \not= 0$ too, as required.
\end{proof}
\subsection{Simplicity of  groupoid \texorpdfstring{$\cs$}{C*}-algebras}

\medskip

\begin{thm}
\label{thm:c*simple}
Let $\G$ be a second-countable, locally compact, \'etale groupoid such that $\go$ is Hausdorff.
\begin{enumerate}
\item If $\cs(\G)$ is simple, then $\cs(\G) = \cs_r(\G)$, $\G$ is effective and for every nonzero $a \in \cs_r(\G)$, $\supp(j(a))$ has nonempty interior.
\item If $\cs_r(\G)$ is simple, then $\G$ is minimal.
\item If $\G$ is minimal and effective and for every nonzero $a \in \cs_r(\G)$, $\supp(j(a))$ has nonempty interior, then $\cs_r(\G)$ is simple.
\end{enumerate}
In particular, if $\cs(\G) = \cs_r(\G)$, then $\cs(\G)$ is simple if and only if $\G$ is
minimal, effective and for every nonzero $a \in \cs_r(\G)$, $\supp(j(a))$ has nonempty interior.
\end{thm}
\begin{proof}
If $\cs(\G) \not= \cs_r(\G)$, then the kernel of the
regular representation is nontrivial and so $\cs(\G)$ is not simple. Suppose that $\G$ is
not effective. Then there is an open bisection $B \nsubseteq \go$
such that for every $\gamma \in B$, $s(\gamma)=r(\gamma)$. Fix $u \in
s(B)$, and let $\epsilon_u$ be the augmentation representation on $\ell^2(r(\G_u))$
described in \cite[Proposition~5.2]{BCFS}. Fix a function $f \in C_c(B)$ and let $f_0 \in
C_c(s(B))$ be the function defined by $f_0(s(\gamma)) = f(\gamma)$ for $\gamma \in B$.
Then $\epsilon_u(f_0)$ is nonzero, and $\epsilon_u(f_0 - f) = 0$. So $\ker\epsilon_u$ is
a nontrivial ideal of $\cs(\G)$.  So we have $\cs(\G) = \cs_r(\G)$, $\G$ is effective.

By way of contradiction, suppose there exists $a \in \cs_r(\G)$ such that $\supp(j(a))$ has empty interior.  That is, $j(a)$ is a singular element.  Then by Proposition~\ref{prop:simpleconverse} $\cs_r(\G)$ is not simple, which is a contradiction.

For~(2) we proceed by contrapositive. Suppose that $\G$ is not minimal.  Then
we can find $u \in \go$ such that $r(\G_u)$ is not dense in $\go$. So there is an open set
$U \subseteq \go$ such that $r(\G_u) \cap U = \emptyset$. For any $f \in C_c(U) \setminus
0$, we have $L_u(f) = 0$, so $L_u$ is a nonzero representation of $\cs_r(\G)$ with
nontrivial kernel, and therefore $\cs_r(\G)$ is not simple.

For~(3), suppose that $I$ is a nonzero ideal of $\cs_r(\G)$. By Theorem~\ref{thm:reduced
uniqueness}, there exists
\[f \in (C_0(\go) \cap I) \setminus \{0\}.\]
For each $u \in \go$,
we have $r(\G_u) \cap \supp(f) \not= \emptyset$, and so we can choose a compact open
bisection $B$ such that $u \in s(B)$, and the unique $\gamma \in B$ with $s(\gamma) = u$
satisfies $f(r(\gamma)) \not= 0$. Fix $h \in C_c(B)$ such that $h(\gamma) = 1$. Then $g
:= h^* f h \in C_0(g_0)$ belongs to $I$ and satisfies
\[g(u) = h^*(\gamma^{-1})f(r(\gamma)) h(\gamma)  = f(r(\gamma)) \not= 0.\]
So $I \cap C_0(\go)$ is an ideal and
there is no $u$ in $\go$ such that $f(u) = 0$ for all $f \in I \cap C_0(\go)$. This
forces $I \cap C_0(\go) = C_0(\go)$, and since $C_0(\go)$ contains an approximate
identity for $\cs_r(\G)$, we deduce that $I = \cs_r(\G)$.
\end{proof}

\subsection{A special case:  simplicity of ample groupoid \texorpdfstring{$\cs$}{C*}-algebras}
\label{csample}

\medskip

In many of the examples we are interested in, the groupoid is ample.   So we state explicitly what our
$\cs$-algebraic results say in this case.   We also show that, like in the Steinberg algebra situation,  the condition of every compact open set is regular open is sufficient.

\begin{lemma}
\label{lem:amplecondequiv}
Let $\G$ be an ample groupoid such that $\go$ is Hausdorff.  If every compact open set is regular open, then for every nonzero $a \in \cs_r(\G)$, $\supp(j(a))$ has nonempty interior.
\end{lemma}

\begin{proof}
Suppose that every compact open set is regular open, i.e.,
that Lemma~\ref{lem:key}\eqref{lemit1:key} holds. Thus
Lemma~\ref{lem:key}\eqref{lemit2:key} also holds.
Fix $a \in \cs_r(\G) \setminus \{0\}$. Since
$A_\CC(\G)$ is dense in $\cs_r(\G)$, we can find $f \in A_\CC(\G)$ where
\[\|f - a\| <\|j(a)\|_\infty/3.\] Since $j$ is norm-decreasing and $j(f) = f$ we deduce that
\[\|f -j(a)\| < \|j(a)\|_\infty/3,\] and hence $\|f\|_\infty > 2\|j(a)\|_\infty/3$. Fix $\gamma$
with $|f(\gamma)| > 2\|j(a)\|_\infty/3$. By Lemma~\ref{lem:key}\eqref{lemit2:key}, the set $O := \{\eta
\in \G : f(\eta) = f(\gamma)\}$ has nonempty interior. Since $\|f - j(a)\|_\infty <
\|j(a)\|_\infty/3$, for each $\eta \in O$, we have
\[
|j(a)(\eta)| \ge |f(\eta)| - \|j(a)\|_\infty/3 > \|j(a)\|_\infty/3 > 0.
\]
So $O \subseteq \{\eta \in \G : j(a)(\eta) \not=0\}$, and so the latter has nonempty
interior because $O$ does.

\end{proof}

Thus we combine Lemma~\ref{lem:amplecondequiv} and Theorem~\ref{thm:c*simple} to get the following corollary.

\begin{cor}
\label{cor:C*simpleample}
Let $\G$ be a second-countable, ample groupoid.
\begin{enumerate}
\item If $\cs(\G)$ is simple, then $\cs(\G) = \cs_r(\G)$ and $\G$ is
minimal and effective.
\item If $\G$ is minimal and effective and every compact open subset of $\G$ is regular
    open, then $\cs_r(\G)$ is simple.
    \item If $\cs_r(\G)$ is simple, then $A_{\C}(\G)$ is simple.
\end{enumerate}
\end{cor}

\section{Examples}
\label{sec:ex}

\subsection{A class showing minimal, topologically principal, and second countable are not sufficient for simplicity}
\label{aidanex}

In this example, we exhibit a class of minimal, topologically principal, amenable, second countable, \'etale groupoids whose corresponding algebras are not simple.
This shows that one really does need the groupoid to be effective and not just topologically principal.

Let $X$ be a compact Hausdorff space with no isolated points, let $\vp:X\to X$ be a minimal homeomorphism (which we recall means that, in contrast to the definition of minimality for a groupoid, the {\em forward} orbit of every point is dense), and fix
$x_0\in X$.
Let $\G = X \sqcup \{a_n: n\in \Z\}$.
We make $\G$ into a groupoid by declaring that the unit space of $\G$ is $X$, that $r(a_n) = s(a_n) = \vp^n(x_0)$ and that $a_na_n = r(a_n)$.  Note that the minimality of $\vp$ means that there are no other composable pairs in $\G \setminus \go$.
The basic open neighbourhoods of $a_n$ are of the form $U \cup \{a_n\} \setminus \{r(a_n)\} $, where $U$ ranges over a base of open neighbourhoods of $r(a_n)$.
It is straightforward to verify that $\G$ is an \'etale groupoid.

The map $\alpha: \G\to \G$ defined by
\[
\alpha(x) =\begin{cases} \vp(x) & \text{if }x\in X\\
a_{n+1}& \text{if }x = a_n\text{ for some }n
\end{cases}
\]
is readily verified to be a topological groupoid isomorphism, and so induces an action of $\Z$ on $\G$.
Recall that the semidirect product $\G\rtimes_\alpha \Z$ is is a groupoid that is the product space (with product topology) $\G \times \Z$ with range, source, product, and inverse given by
\[
r(\gamma, n) = (r(\gamma), 0),\hspace{1cm} s(\gamma, n) = (\alpha^{-n}(s(\gamma)), 0),
\]
\[
 (\gamma, n)(\alpha^{-n}(\eta), m) = (\gamma\eta, m+n),\hspace{1cm} (\gamma, n)^{-1} = (\alpha^{-n}(\gamma^{-1}), -n).
\]
Then $\G\rtimes_\alpha\Z$ is \'etale.
Since $X$ has no isolated points, $\G$ is not Hausdorff, and so neither is $\G\rtimes_\alpha\Z$.
The unit space of $\G\rtimes_\alpha \Z$ can be identified with $X$, and since $\vp$ is minimal, $\G\rtimes_\alpha \Z$ is minimal.
If $X$ is totally disconnected, then $\G\rtimes_\alpha \Z$ is ample.

If $r(\gamma,n) = s(\gamma, n)$, we must have that $\alpha^{-n}(s(\gamma))= r(\gamma)$.
Since $s(\gamma) = r(\gamma) \in X$ for all $\gamma\in \G$, this implies that $r(\gamma)$ is periodic for $\vp$ which contradicts minimality, unless $n = 0$. Hence $\operatorname{Iso}(\G\rtimes_\alpha\Z)\setminus X = \{(a_n,0): n\in\Z\}$ and each $(a_n,0)$
has an open neighbourhood of the form $V\times \{0\}$ contained in
$\operatorname{Iso}(\G\rtimes_\alpha\Z)$.
Hence, $\G\rtimes_\alpha\Z$ is not effective.
Take $y\in X$ not in the orbit of $x_0$ (which must exist because compact Hausdorff spaces with no isolated points must be uncountable).
Then the orbit of $y$ is dense in $X$, and each point in the orbit of $y$ has trivial isotropy.
Hence $\G\rtimes_\alpha\Z$ is topologically principal.

Consider the functions
\[
f = 1_{\G\rtimes_\alpha\Z^{(0)}}
\]
\[
g = 1_{\G\rtimes_\alpha\Z^{(0)}\cup\{(a_0,0)\}\setminus\{(x_0,0)\}}.
\]
Then $f,g\in \Cc(\G)$, since both $\G\rtimes_\alpha\Z^{(0)}$ and $\G\rtimes_\alpha\Z^{(0)}\cup\{(a_0,0)\}\setminus\{(x_0,0)\}$ are compact open bisections in $\G\rtimes_\alpha\Z$ (they are both homeomorphic to $X$).
But $\supp(f-g) = \{(a_0,0)\}\cup\{(x_0,0)\}$, which has empty interior.
Hence by Theorem~\ref{thm:c*simple}, $\cs_r(\G\rtimes_\alpha \Z) = \cs(\G\rtimes_\alpha \Z)$ is not simple.
Furthermore, if $\G\rtimes_\alpha\Z$ is ample,  then $f,g$ are elements of the Steinberg algebra.
So neither $A_\KK(\G\rtimes_\alpha\Z)$ nor $\cs_r(\G\rtimes_\alpha\Z)$ is simple by Theorem~\ref{thm:simple} and
Theorem~\ref{thm:c*simple} respectively.

\subsection{Inverse semigroup actions and their groupoids}

For any unreferenced claims in this section, see \cite{Exel3} and \cite{EP16}.
We will use the notation $Y\fs X$ to indicate that $Y$ is a finite subset of $X$.

An {\em inverse semigroup} is a semigroup $S$ for which every $s\in S$ has a unique ``inverse'' $s^*$ in the sense that
\[
ss^*s= s \quad \text{and} \quad s^*ss^* = s^*.
\]
For every $s, t\in S$ we have $(s^*)^* = s$ and $(st)^* = t^*s^*$. If $S$ has an identity, we will denote it $1_S$.
Every inverse semigroup will be assumed to be countable and have a zero element 0 which satisfies $0s = s0 = 0$ for all $s\in S$. The set of {\em idempotents} of $S$ is denoted
\[
E(S) = \{e\in S: e^2 = e\}
\]
and consists of all elements of the form $s^*s$.

Any inverse semigroup carries a natural order structure. For $s,t\in S$, we write $s\leqslant t$ if $ts^*s = s$. For two idempotents $e, f\in E(S)$, we have $e\leqslant f$ if and only if $ef = e$. A {\em filter} in $E(S)$ is a nonempty subset $\xi\subseteq
E(S)$ such that
\begin{enumerate}
	\item$0\notin \xi$,
	\item$e, f\in \xi$ implies that $ef\in \xi$, and
	\item $e\in \xi, e\leqslant f$ implies $f\in \xi$.
\end{enumerate}
We denote the set of filters $\Ef(S)$; it can be viewed as a subspace of $\{0,1\}^{E(S)}$. For $X,Y\fs E(S)$, let
\begin{equation}
U(X,Y) = \{\xi\in \Ef(S): X\subseteq \xi, Y\cap \xi = \emptyset\}.
\end{equation}
Sets of this form are clopen and generate the topology on $\Ef(S)$ as $X$ and $Y$ vary over all the finite subsets of $E(S)$.
With this topology, $\Ef(S)$ is called the {\em spectrum} of $E(S)$.
From the definition of a filter it is easy to see that if $X, Y\fs E(S)$ and $e := \prod_{x\in X}x$, then $U(X,Y) = U(\{e\}, Y)$, and so we can take sets of the form
\begin{equation}\label{eq:filterbasis}
U(\{e\},Y) = \{\xi\in \Ef(S): e\in \xi, Y\cap \xi = \emptyset\}
\end{equation}
to be the basis of the topology on $\Ef(S)$.

A filter is called an {\em ultrafilter} if it is not properly contained in any other filter. The set of all ultrafilters is denoted $\Eu(S)$. As a subspace of $\Ef(S)$, $\Eu(S)$ may not be closed. Let $\Et(S)$ denote the closure of $\Eu(S)$ in $\Ef(S)$ ---
this is called the {\em tight spectrum} of $E(S)$.

An {\em action} of an inverse semigroup $S$ on a space $X$ consists of a collection $\alpha = \{\alpha_s\}_{s\in S}$ of homeomorphisms between open subsets of $S$ satisfying:
\begin{enumerate}
	\item $\alpha_s\circ \alpha_t = \alpha_{st}$ for all $s, t\in S$, and
	\item the union of the domains of the $\alpha_s$ coincides with $X$.
\end{enumerate}
These imply that if $e$ is an idempotent, then $\alpha_e$ is the identity map on some open subset $D_e\subseteq X$, and that the domain of $\theta_s$ coincides with $D_{s^*s}$.

If $\alpha$ is an action of $S$ on a space $X$, we let
\[
S\ltimes_\alpha X = \{(s, x): x\in D_{s^*s}\}
\]
and put an equivalence relation on this set by saying $(s, x)\sim (t, y)$ if and only if $x = y$ and there exists $e\in E(S)$ such that $x\in D_e$ and $se = te$. The equivalence class of $(s, x)$ is denoted $[s, x]$ and is called the {\em germ} of $(s,x)$.
The set of all germs is denoted
\[
\G(\alpha) = \{[s,x]: x\in D_{s^*s}\}
\]
and becomes a groupoid with source, range, inverse, and product given by
\[
s([t,x]) = [t^*t, x], \quad r([t,x]) = [tt^*, \alpha_t(x)], \quad [t, x]^{-1} = [t^*, \alpha_t(x)], \quad [u, \alpha_t(x)][t,x] = [ut, x].
\]
The unit space $\G(\alpha)^{(0)}$ is identified with $X$.
If $X$ is a totally disconnected locally compact Hausdorff space such that $D_e$ is compact open for all $e\in E(S)$,  then $\G(\alpha)$ is ample, with sets of the form
\begin{equation}
\Theta(s, U) := \{[s, x]\in \G(\alpha): x\in U\subseteq D_{s^*s}\text{ and }U\text{ compact open}\}
\end{equation}
forming a basis of compact open sets for $\G(\alpha)$.
If $S$ is countable, then $\G(\alpha)$ is second countable.

It is possible for $\G(\alpha)$ to be non-Hausdorff, as we will see in examples of subsections \ref{sec:ssg} and \ref{sec:grig}.
In \cite[Theorem 3.15]{EP16} are given criteria on $\alpha$ which are equivalent to $\G(\alpha)$ being Hausdorff.
This problem was also considered in \cite{St2}.

An inverse semigroup acts on its tight spectrum. Let
\[
D^\theta_e = \{\xi\in \Et(S): e\in \xi\} = U(\{e\},\varnothing),
\]
and define
\[
\theta_s: D^\theta_{s^*s}\to D^\theta_{ss^*}
\]
\[
\theta_s(\xi) = \{e\in E(S): e\geqslant sfs^*\text{ for some }f\in \xi\}.
\]
This defines an action of $S$ on $\Et(S)$, called the {\em standard action} of $S$.
The groupoid of germs for the standard action is denoted $\gt(S) := \G(\theta)$ and is called the {\em tight groupoid} of $S$.
This groupoid is quite general.
Indeed, by \cite{Exel2} every ample groupoid arises this way.

In what follows, we are concerned with the subsets \eqref{eq:filterbasis} intersected with $\Et(S)$.
If $\xi$ is an ultrafilter, then by \cite[Proposition 2.5]{EP16} the set $\{D^\theta_e: e\in \xi\}$ is a neighbourhood basis for $\xi$.
Hence, if every tight filter is an ultrafilter (i.e., if $\Et(S) = \Eu(S)$), then the $D_e$ form a basis for the topology on $\Et(S)$.

\begin{lemma}\label{lem:basisISG}
	Let $S$ be an inverse semigroup and suppose that $\Et(S) = \Eu(S)$.
	Then
	\[
	\{\Theta(s, D^\theta_{s^*s}) :  s\in S\}\cup\{\emptyset\}
	\]
	is an inverse semigroup of compact open bisections which generates the topology of $\gt(S)$.
\end{lemma}
\begin{proof}
	The given set generates the topology on $\gt(S)$ because the $D_e^\theta$ generate the topology of $\Et(S)$ when $\Et(S) = \Eu(S)$.
	Since the set of compact open bisections in an ample groupoid forms an inverse semigroup under setwise product and inverse, we need to prove our set is closed under the product and inverse.
	
	We claim that
	\[
	\Theta(s, D^\theta_{s^*s})\Theta(t, D^\theta_{t^*t}) = \Theta(st, D^\theta_{(st)^*st})
	\]
	if the product is nonempty.
	If $\gamma\in \Theta(s, D^\theta_{s^*s})\Theta(t, D^\theta_{t^*t})$, then $\gamma = [s, \theta_t(\eta)][t, \eta]$ for some $\eta\in D^\theta_{t^*t}$.
	Since $\theta_t(\eta)\in D^\theta_{s^*s}$, we must have $\eta = \theta_{t^*}(\xi)$ for some $\xi \in D^\theta_{s^*s}$.
	Because $s^*s\in \xi$, $t^*s^*st = (st)^*st\in \eta$, and so $\eta\in D^\theta_{(st)^*st}$.
	Thus $\gamma\in \Theta(st, D^\theta_{(st)^*st})$ proving the $\subseteq$ inclusion.
	
	To prove the $\supseteq$ inclusion, we take $\xi\in D^\theta_{(st)^*st}$ and notice that $[st, \xi] = [s, \theta_t(\xi)][t,\xi]$.
	
	To prove closure under inverse, it is straightforward that
	\begin{align*}
	\Theta(s, D^\theta_{s^*s})^{-1} &= \{[s,\xi]^{-1}:\xi\in D^\theta_{s^*s} \}=  \{[s^*,\theta_s(\xi)]:\xi\in D^\theta_{s^*s} \}\\ &= \{[s^*,\eta]:\eta\in D^\theta_{ss^*} \} = \Theta(s^*, D^\theta_{ss^*})
	\end{align*}
	because $\theta_s: D^\theta_{s^*s}\to D^\theta_{ss^*}$ is a bijection.
\end{proof}

\begin{lemma}\label{lem:inclosure}
	Let $S$ be a inverse semigroup, let $s, t\in S$, let $\xi\in D^\theta_{t^*t}$, let $e\leqslant s^*s$, and let $Y\fs E(S)$. Then:
	\begin{enumerate}
		\item The following are equivalent:
		\begin{enumerate}
			\item $[t, \xi]\in \overline{\Theta(s, U(\{e\}, Y))}$
			\item $\xi\in  U(\{e\}, Y)$ and for every $f\in \xi$ and for all $Z\fs E(S)\setminus \xi$, there exists $0\neq k\leqslant ef$ such that $ky = 0$ for all $y\in Y\cup Z$, and $sk = tk$.
		\end{enumerate}
		\item In case of $(1)$, $[t, \xi]\not\in {\Theta(s, U(\{e\}, Y))}$ if and only if for any such $k$ we have that $k\not\in \xi$.
	\end{enumerate}
\end{lemma}
\begin{proof}
	For the sake of brevity, define $R:=\Theta(s, U(\{e\}, Y))$ and suppose that $[t, \xi]\in \overline{R}$.
	We have
	$$\xi= s([t,\xi])\in s(\overline R)\subseteq \overline{s(R)} = \overline{U(\{e\}, Y)}.$$
	Since $U(\{e\}, Y)$ is closed in $\Et(S)$, there exists a closed set $V\subseteq \gt(S)$ such that $U(\{e\}, Y) = V\cap \Et(S)$.
	Hence, $\overline{U(\{e\}, Y)} \subseteq \overline V = V$ and so $\overline{U(\{e\}, Y)}\cap \Et(S)\subseteq V\cap \Et(S) = U(\{e\}, Y)$, and so $\xi\in U(\{e\}, Y)$.
	
	For all $f\in \xi$ and for all $Z\fs E(S)\setminus \xi$, there is a point in the intersection $\Theta(s, U(\{e\}, Y))\cap \Theta(t, U(\{f\}, Z))$.
	Hence for all $f\in \xi$ and for all $Z\fs E(S)\setminus \xi$, there exists $\eta\in U(\{e\}, Y)\cap U(\{f\}, Z) = U(\{ef\}, Y\cup Z)$ such that $[s,\eta] = [t,\eta]$.
	In other words, there exists $k\in \eta$ such that $tk = sk$.
	Since $\Theta(s, U(\{e\}, Y))\cap \Theta(t, U(\{f\}, Z))$ is open and $s$ is an open map, we may assume that $\eta$ is an ultrafilter.
	Without loss of generality, we can assume that $k\leqslant ef$ by perhaps replacing it with $kef$ if needed, and since $k, e, f\in \eta$, we must have $kef\neq 0$.
	If $y\in Y\cup Z$, the fact that $\eta$ is an ultrafilter which does not contain $y$ implies there exists $e_y\in \eta$ such that $ye_y = 0$.
	So also without loss of generality, we can assume that $ky = 0$ by perhaps replacing it by $ke_y$ if needed.
	This establishes (a) $\Rightarrow$ (b)
	
	For the other implication, suppose (b) holds.
	We will be done if we show that for every $f\in \xi$ and for all $Z\fs E(S)\setminus \xi$, there is a point in $\Theta(s, U(\{e\}, Y))\cap \Theta(t, U(\{f\}, Z))$.
	For $f\in \xi$, find the $k\leqslant f$ guaranteed by (b) and find an ultrafilter $\eta$ containing $k$.
	Since $tk = sk$, we have that $[t,\eta]=[s,\eta]$, and since $k\in \eta$ we must have that $k,e, f\in y$.
	Since $ky = 0$ for all $y\in Y\cup Z$ we must also have that $y\notin \eta$ for all $y\in Y\cup Z$.
	Hence $\Theta(t, U(\{f\}, Z))\ni [t, \eta] = [s,\eta] \in \Theta(s, U(\{e\}, Y))$, establishing (b) $\Rightarrow$ (a).
	
	Point (2) is direct, so we are done.
\end{proof}

\begin{lemma}\label{lem:Effective=> basic reg}
	Let $S$ be a inverse semigroup, and take $s\in S$, $e\in E(S)$ with $e\leqslant s^*s$, and $Y\fs E(S)$. If $\gt(S)$ is effective, then $\Theta(s, U(\{e\}, Y))$ is regular open.
\end{lemma}
\begin{proof}
	This follows from Lemma~\ref{lem:bisections reg open}, because each $\Theta(s, U(\{e\}, Y))$ is a compact open bisection.
\end{proof}

Let us introduce a couple of (equivalent) properties that extend effectiveness of $\gt(S)$, and give us necessary conditions for regularity of compact open sets. Following the notation in \cite[Section 4]{EP16}, for any $s\in S$ we will denote by
$$F_s:=\{x\in \Et(S) : \theta_s(x)=x \}$$
the set of {\em fixed elements} for $s$, and by
$$TF_s:=\{x\in \Et(S) : \exists e\in E(S) \text{ with } e\leqslant s \text{ and } x\in D^\theta_e\}$$
the set of {\em trivially fixed elements} for $s$. By \cite[Proposition 4.5]{EP16}, if $\mathcal{J}_s:=\{e\in E :e\leqslant s\}$, then
$$TF_s=\bigcup\limits_{e\in \mathcal{J}_s}D^\theta_e$$

\begin{definition}\label{Def:cond(S)}
	Let $S$ be an inverse semigroup.
	We will say $S$ satisfies \emph{condition \eqref{eq:cond(S)}} if whenever there exists an element $x\in \Et(S)$ and a finite set $\{s_1, \dots, s_n\}\subseteq S\setminus E(S)$ such that  $x\in \bigcap\limits_{i=1}^{n}F_{s_i}$, then
	
	\begin{gather}\label{eq:cond(S)}\tag{S}
	x\in \bigcap\limits_{i=1}^{n}(F_{s_i}\setminus TF_{s_i}) \implies x\not\in \Big(\bigcup_{i = 1}^nF_{s_i}\Big)^\circ .
	\end{gather}
	
\end{definition}

\begin{remark}\label{Rem:various}
	Suppose that $S$ satisfies \eqref{eq:cond(S)}.
	For the particular case $n=1$ we have that $x\in F_s\setminus TF_s$ implies $x\not\in F^\circ_{s}$; this is equivalent to saying that if $x\in F^\circ_{s}$ then $x\not\in F_s\setminus TF_s$, i.e. $x\in TF_s$, which is exactly \cite[Theorem
4.10(iii)]{EP16}.
	Thus if either $\Eu(S)=\Et(S)$ or $\gt(S)$ is Hausdorff and $S$ satisfies \eqref{eq:cond(S)},  \cite[Theorem 4.10]{EP16} implies $\gt(S)$ is effective.
	We will see in the sequel that effectiveness of $\gt(S)$ does not imply \eqref{eq:cond(S)}, so for $n=1$ \eqref{eq:cond(S)} is closely related to but in general weaker than effectiveness of $\gt(S)$.
	
which is exactly \cite[Theorem 4.10(iii)]{EP16}.
	
in general closely related but weaker than effectiveness of $\gt(S)$.
\end{remark}

\begin{lemma}\label{Lem: (E) => regular}
	Let $S$ be an inverse semigroup which satisfies condition \eqref{eq:cond(S)}, and suppose that $\Et(S) = \Eu(S)$.  Then any compact open subset of $\gt(S)$ is regular open.
\end{lemma}
\begin{proof}
	As noted before, the condition $\Et(S) = \Eu(S)$ implies that the $D^\theta_e$ are a basis for the topology on $\Et(S)$, and hence sets of the form $\Theta(t, D^\theta_e)$ are a basis for the topology on $\gt(S)$.
	
	Let $V$ be a compact open  subset of $\gt(S)$.
	Then, there exist $\{s_1, \dots, s_n\}\subseteq S$ and $\{e_1, \dots ,e_n\}\subseteq E(S)$ with $e_i\leqslant s_i^*s_i$ for all $1\leq i\leq n$ such that
	$$V=\bigcup\limits_{i=1}^{n}\Theta(s_i, D^\theta_{e_i}).$$
	Take $[t,x]\in \overline{V}\setminus V$.
	We will show that any open neighbourhood of $[t, x]$ contains a point outside of $\overline{V}$.
	If $1\leq i\leq n$ is any index such that $[t,x]\notin \overline{\Theta(s_i, D^\theta_{e_i})}$, then there is a neighbourhood of $[t,x]$ disjoint from $\Theta(s_i, D^\theta_{e_i})$, so we can assume without loss of generality that $[t,x]\in
\overline{\Theta(s_i, D^\theta_{e_i})}\setminus \Theta(s_i, D^\theta_{e_i})$ for all $1\leq i\leq n$.
	Hence, by Lemma \ref{lem:inclosure}, for each $1\leq i\leq n$ we have that $\theta_{t^*s_i}(x)=x$ but $[t^*s_i, x]\not\in T{F}_{t^*s_i}$; in particular, $\{t^*s_1, \dots ,t^*s_n\}\subseteq S\setminus E(S)$.
	Thus, $x\in \bigcap\limits_{i=1}^{n}(F_{t^*s_i}\setminus T{F}_{t^*s_i})$.
	Since condition \eqref{eq:cond(S)} holds, for every $f\in x$ we have that $D^\theta_f\not\subseteq \bigcup\limits_{i=1}^{n}F_{t^*s_i}$, and so $\Theta(t,D^\theta_f)\not\subseteq\overline{V}$, as desired.
\end{proof}

Now, we will apply these results to various classes of algebras.

\subsection{Algebras of self-similar graphs} \label{sec:ssE}

In this subsection, we consider the algebras $\OGE$ associated to triples $(G,E, \varphi)$, introduced in \cite{EP17}.   We use the convention where a path in the graph $E$ is a sequence of edges
$e_1\cdots e_n$ such that $s(e_i)=r(e_{i+1})$.  Let us recall the construction.

\begin{noname}\label{basicdata}
	{\rm The basic data for our construction is a triple $(G,E,\varphi)$ consisting of:
		\begin{enumerate}
			\item A finite directed graph $E=(E^0, E^1, r, s)$ without sources.
			\item A discrete group $G$ acting on $E$ by graph automorphisms.
			\item A map $\varphi: G\times E^1\rightarrow G$ satisfying
			\begin{enumerate}
				\item $\varphi(gh, a) = \varphi(g, h\cdot a)\varphi(h, a)$, and
				\item $\varphi(g,a)\cdot v=g\cdot v$ for every $g\in G, a\in E^1, v\in E^0$.
			\end{enumerate}
		\end{enumerate}
	}
\end{noname}
The property~(3)(b) required of $\varphi$ is tagged~$(2.3)$ in \cite{EP17}.

\begin{definition}\label{Def:OGE}
	{\rm Given a triple $(G,E, \varphi)$ as in~(\ref{basicdata}), we define $\OGE$ to be the universal $\cs$-algebra as follows:
		\begin{enumerate}
			\item \underline{Generators}:
			$$\{p_x : x\in E^0\}\cup\{s_a : a\in E^1\} \cup \{u_g : g\in G\}.$$
			\item \underline{Relations}:
			\begin{enumerate}
				\item $\{p_x : x\in E^0\}\cup\{s_a : a\in E^1\}$ is a Cuntz-Krieger $E$-family in the sense of \cite{Raeburn}.
				\item The map $u:G\rightarrow \OGE$ defined by the rule $g\mapsto u_g$ is a unitary $\ast$-representation of $G$.
				\item $u_gs_a=s_{g\cdot a}u_{\varphi(g,a)}$ for every $g\in G, a\in E^1$.
				\item $u_gp_x=p_{g\cdot x}u_g$ for every $g\in G, x\in E^0$.
			\end{enumerate}
		\end{enumerate}
	}
\end{definition}
\noindent Notice that the relation (2a) in Definition \ref{Def:OGE} implies that there is a natural representation map
$$
\begin{array}{cccc}
\phi: & \cs(E) &\to   & \OGE  \\
& p_x & \mapsto  & p_x \\
& s_a & \mapsto  & s_a
\end{array}
$$
which is injective \cite[Proposition 11.1]{EP17}.

Recall from \cite[Definition 4.1]{EP17} that given a triple $(G,E,\varphi)$ as in (\ref{basicdata}), we define an inverse semigroup $\SGE$ as follows:
\begin{enumerate}
	\item The set is
	$$\SGE=\{ (\alpha,g,\beta) : \alpha, \beta\in E^*, g\in G, d(\alpha)=gd(\beta)\}\cup \{ 0\},$$
	where $E^*$ denotes the set of finite paths in $E$.
	\item The operation is defined by:
	$$(\alpha,g,\beta)(\gamma,h,\delta):=
	\begin{cases}
	(\alpha (g\cdot\varepsilon), \varphi(g,\varepsilon) h, \delta) &  \text{if } \gamma=\beta\varepsilon    \\
	&\\
	(\alpha, g\varphi(h^{-1},\varepsilon)^{-1}, \delta (h^{-1}\cdot\varepsilon)) &   \text{if } \beta=\gamma \varepsilon  \\
	&\\
	0 & \text{otherwise,}
	\end{cases}
	$$
	and $(\alpha,g,\beta)^*:= (\beta,g^{-1}, \alpha)$.
\end{enumerate}

Then, we can construct the groupoid of germs of the action of $\SGE$ on the space of tight filters $\Et(\SGE)$
of the semilattice ${E}(\SGE)$ of idempotents of $\SGE$. In our concrete case, $\Et(\SGE)$ turns out to be
homeomorphic to the compact space $E^{\infty}$ of one-sided infinite paths on $E$ which has the {\em cylinder sets}
$$C(\gamma):=\{\gamma \widehat{\eta} : \widehat{\eta}\in E^{\infty}\}$$ as a basis for its topology.
In particular, $\Et(\SGE)=\Eu(\SGE)$, and thus condition \eqref{eq:cond(S)} implies effectiveness of the groupoid $\CG$ defined below.

The action of
$(\alpha,g,\beta)\in \SGE$ on $\eta=\beta\widehat{\eta}$  is given by the rule $(\alpha,g,\beta)\cdot \eta=\alpha (g\widehat{\eta})$.
Thus, the groupoid of germs is
$$\CG=\{[\alpha,g,\beta;\eta] : \eta=\beta\widehat{\eta}\},$$
where $[s;\eta]=[t;\mu]$ if and only if $\eta=\mu$ and there exists $0\ne e^2=e\in \SGE$ such that $e\cdot\eta=\eta$ and $se=te$.
The unit space
\[\CG^{(0)}=\{[\alpha,1,\alpha;\eta] : \eta=\alpha\widehat{\eta}\}\]
is identified with the one-sided infinite path space $E^{\infty}$, via the
homeomorphism $[\alpha,1,\alpha; \eta]   \mapsto   \eta.$
Under this identification, the range and source maps on
$\CG$ are:
\[s([\alpha,g,\beta;\beta \widehat{\eta}])= \beta\widehat{\eta} \quad \text{ and } \quad
r([\alpha,g,\beta;\beta \widehat{\eta}])= \alpha(g \widehat{\eta}).\]

A basis for the topology on $\CG$ is given by compact open bisections of the form
\[ \Theta(\alpha,g,\beta,C(\gamma)):= \{ [\alpha, g, \beta;\xi ] \in \CG : \xi \in  C(\gamma) \} \]
for $\alpha,\beta,\gamma \in E^*$ and $g\in G$.
Thus $\CG$ is locally compact and ample.
In \cite{EP17} characterizations are given for when
$\CG$ is Hausdorff \cite[Theorem 12.2]{EP17},
amenable \cite[Corollary 10.18]{EP17} and effective \cite[Theorem 14.10]{EP17}\footnote{We note that in \cite{EP16} and \cite{EP17}, `essentially principal' is defined to be what we are calling effective, see \cite[Definition 4.6]{EP16}.}
in terms of the properties of the triple $(G,E,\varphi)$ and the action of $\SGE$ on $E^{\infty}$.

We recall the property which guarantees that $\CG$ is Hausdorff.
For this we require some notation.
For $g\in G$, let
\begin{equation}\label{eq:FWdef}
FW_g = \{\alpha\in X^*: g\cdot \alpha = \alpha\}
\end{equation}
and call this the set of {\em fixed} paths for $g$.
We also let
\begin{equation}\label{eq:SFWdef}
SFW_g = \{\alpha\in E_A^*: g\cdot \alpha = \alpha\text{ and }\vp(g, \alpha)= 1_G\}
\end{equation}
and call this the set of {\em strongly fixed} paths for $g$.
If $\beta\in E_A^*$ has a prefix which is strongly fixed by $g$, then $\beta$ will be strongly fixed by $g$ as well.
We say a path $\alpha$ is {\em minimally strongly fixed} by $g$ if it is strongly fixed by $g$ and no proper prefix of $\alpha$ is strongly fixed by $g$.
We denote this set by
\begin{equation}\label{eq:MSFWdef}
MSFW_g = \{\alpha\in E_A^*: \alpha\in SFW_g\text{ and no prefix of $\alpha$ is in }SF_g\}.
\end{equation}

In \cite[Theorem 12.2]{EP17} it is shown that
\begin{equation}\label{eq:SSGHausdorff}
\CG \text{ is Hausdorff } \iff MSFW_g \text{ is finite for all }g\in G\setminus\{1_G\}.
\end{equation}

By \cite[Theorem 6.3 \& Corollary 6.4]{EP17}, we have a $\ast$-isomorphism $\OGE\cong \cs(\CG)$, so that $\OGE$ can be seen as a full groupoid $\cs$-algebra. The complex Steinberg algebra $A_{\mathbb{C}}(\CG)$ is a dense subalgebra of $\mathcal{O}_{G,E} \cong
\cs(\CG)$
by \cite[Proposition~6.7]{St}.

Finally, recall that for any unital commutative ring $R$,
the Steinberg algebra $A_R(\CG)$ is isomorphic to the $R$-algebra
$\mathcal{O}_{(G,E)}^{\text{alg}}(R)$ with presentation given by Definition \ref{Def:OGE} \cite[Theorem 6.4]{CEP}.

Now, we apply the results obtained in the previous subsection to this case.

\begin{lemma}\label{lem:Effec=> basic reg SGE}
	Let $(G,E,\varphi)$ be a triple as in (\ref{basicdata}), let $s\in \SGE$ and $e\in E(\SGE)$ with $e\leqslant s^*s$. If $\CG$ is effective, then $\Theta(s, D^\theta_e)$ is regular open.
\end{lemma}
\begin{proof}
	This is immediate from Lemma \ref{lem:Effective=> basic reg}.
\end{proof}

We next prove a lemma that will allow us to verify condition \eqref{eq:cond(S)} more readily in this context.

\begin{lemma}\label{lem:SforSSGraphs}
	Let $(G,E,\varphi)$ be a triple as in (\ref{basicdata}), and suppose that for every vertex $v$ and every finite set $\{g_1, g_2, \dots g_n\}\subseteq G\setminus\{1_G\}$ we have that
	\begin{equation}\label{eq:SforSSGraphs}
	x\in \bigcap_{i=1}^nF_{(v,g_i,v)}\setminus TF_{(v,g_i,v)} \implies x\notin \left(
	\bigcup_{i = 1}^n F_{(v,g_i,v)}\right)^\circ.
	\end{equation}
	Then $\SGE$ satisfies \eqref{eq:cond(S)}.
\end{lemma}
\begin{proof}
	Suppose that we have $\{s_1, \dots s_n\}\subseteq \SGE\setminus E(\SGE)$ and $x\in \bigcap_{i = 1}^{n}F_{s_i}\setminus TF_{s_i}$.
	We need to show that $x\notin \left(
	\bigcup_{i = 1}^n F_{s_i}\right)^\circ$.
	
	For $i = 1, \dots n$, write $s_i = (\alpha_i, g_i, \beta_i)$.
	If $|\alpha_i|\neq |\beta_i|$, by \cite[Proposition~14.3]{EP17} $F_{s_i}$ has at most one point, and since $x$ is assumed to be in this set it must be $x$.
	Hence $F_{s_i}\subset F_{s_j}$ for all $j = 1, \dots n$, so in showing that $x\notin \left(
	\bigcup_{i = 1}^n F_{s_i}\right)^\circ$ we can assume without loss of generality that $|\alpha_i| = |\beta_i|$ for $i =1, \dots, n$, which means that $\alpha_i = \beta_i$ for $i = 1, \dots n$.
	
	Since $x$ is fixed by each $s_i$, there exists $\alpha\in E^*$ such that $\alpha = \alpha_i\gamma_i$ for each $i = 1, \dots, n$ and $\gamma_i\in E^*$.
	We also must have that $g_i\cdot \gamma_i = \gamma_i$ and $s(\gamma_i) = s(\alpha)$ for each $i = 1, \dots, n$.
	
	Write $v := s(\alpha)$ and let $t_i = (v, \vp(g_i, \gamma_i), v)$.
	We claim that $\alpha F_{t_i} = F_{s_i}$ and $\alpha TF_{t_i} = TF_{s_i}$ for $i = 1, \dots n$.
	First, let $y\in F_{t_i}$.
	Then $\vp(g_i, \gamma_i)\cdot y = y$.
	We calculate
	\begin{align*}
	\theta_{s_i}(\alpha y) &= \theta_{(\alpha_i, g_i, \alpha_i)}(\alpha_i\gamma_i y)\\
	&= \alpha_i g_i\cdot (\gamma_i y)\\
	&= \alpha_i \gamma_i \vp(g_i, \gamma_i)\cdot y\\
	& = \alpha y
	\end{align*}
	Conversely, if $\alpha z \in F_{s_i}$ a similar calculation shows that $z\in F_{t_i}$, and so $\alpha F_{t_i} = F_{s_i}$.
	
	If $x\in TF_{t_i}$, then there exists $e\in E(\SGE)$ such that $t_ie = e$ and $x\in D_e^\theta$.
	We may write $e = (\mu, 1_G, \mu)$ for some $\mu\in E^*$ with $|\mu|\geq 1$, and so $y = \mu z$ for some $z\in E^\infty$.
	We want to show $\alpha\mu z$ is trivially fixed by $s_i$.
	Since $t_i e = e$, we have
	\begin{align*}
	(v, \vp(g_i, \gamma_i), v)(\mu, 1_G, \mu) &= (\mu, 1_G, \mu)\\
	(\vp(g_i, \gamma_i)\cdot\mu, \vp(\vp(g_i, \gamma_i), \mu), \mu) &= (\mu, 1_G, \mu)
	\end{align*}
	Which implies $\vp(g_i, \gamma_i)\cdot\mu = \mu$ and $\vp(\vp(g_i, \gamma_i), \mu) = 1_G$.
	If we let $f = (\alpha\mu, 1_G, \alpha\mu)$, then
	\begin{align*}
	s_if &= (\alpha_i, g_i, \alpha_i) (\alpha_i\gamma_i\mu, 1_G, \alpha_i\gamma_i\mu) \\
	&= (\alpha_ig_i\cdot(\gamma_i\mu), \vp(g_i, \gamma_i\mu), \alpha\mu))\\
	&= (\alpha_i\gamma_i\vp(g_i, \mu)\cdot \mu, \vp(g_i, \gamma_i\mu)\vp(g_i, \gamma_i\mu), \alpha\mu)\\
	&= (\alpha\mu, 1_G, \alpha\mu) = f
	\end{align*}
	and clearly $\alpha\mu z\in D_f^\theta$.
	Hence $\alpha\mu z$ is trivially fixed by $s_i$.
	Conversely, a similar calculation shows that if $\alpha y$ is trivially fixed by $s_i$ then $y$ is trivially fixed by $t_i$.
	
	Now, $x\in \bigcap_{i = 1}^{n}F_{s_i}\setminus TF_{s_i}$ implies $x = \alpha y$ for some $y\in E^\infty$.
	Since
	$$\bigcap_{i = 1}^{n}F_{s_i}\setminus TF_{s_i} = \bigcap_{i = 1}^{n}\alpha F_{t_i}\setminus \alpha TF_{t_i} = \alpha \bigcap_{i = 1}^{n}F_{t_i}\setminus TF_{ts_i}$$
	we have that $y\in \bigcap_{i = 1}^{n}F_{t_i}\setminus TF_{t_i}$ and so by hypothesis we must have that $y\notin \left(
	\bigcup_{i = 1}^n F_{t_i}\right)^\circ$.
	
	Now suppose $\alpha y\in \left(
	\bigcup_{i = 1}^n \alpha F_{t_i}\right)^\circ$.
	Then there exists a prefix $\mu$ of $y$ such that $\alpha y = \alpha\mu y'$ and $C(\alpha\mu) \subseteq \bigcup_{i = 1}^n \alpha F_{t_i}$.
	But then we must have $y\in C(\mu)\subseteq \bigcup_{i = 1}^n F_{t_i}$, which contradicts the fact that $y\notin \left(
	\bigcup_{i = 1}^n F_{t_i}\right)^\circ$.
	Therefore we must have  $\alpha y = x\notin \left(
	\bigcup_{i = 1}^n \alpha F_{t_i}\right)^\circ$ and we are done.
\end{proof}
\begin{lemma}\label{Lem: (E) => regular SGE}
	Let $(G,E,\varphi)$ be a triple as in \rm{(\ref{basicdata})} which satisfies \eqref{eq:SforSSGraphs}.
	Then, any compact open set is regular open.
\end{lemma}
\begin{proof}
	By Lemma~\ref{lem:SforSSGraphs}, $\SGE$ satisfies \eqref{eq:cond(S)} and so Lemma~\ref{Lem: (E) => regular} implies the result.
\end{proof}

We can now state a consequence of our results to the case of self-similar graphs.
\begin{thm}\label{thm:ssEsimple}
	Let $(G,E,\varphi)$ be a triple as in~\eqref{basicdata} such that $\SGE$ satisfies condition~\eqref{eq:cond(S)} and such that $\CG$ is minimal.
	Then:
	\begin{enumerate}
		\item $\cs_r(\CG)$ is simple.
		\item If $G$ is amenable, then $\OGE$ is simple.
		\item For any field $\KK$, the algebra $A_\KK(\CG)$ is simple.
	\end{enumerate}
\end{thm}
\begin{proof}
	Since $\Et(\SGE)=\Eu(\SGE)$, condition \eqref{eq:cond(S)} and countability implies that $\CG$ is effective by Remark \ref{Rem:various}.
	By Lemma~\ref{Lem: (E) => regular}, every compact open subset of $\CG$ is regular open. Now (3) follows from Theorem~\ref{thm:simple}, and (1) and (2) follow from Theorem~\ref{thm:c*simple} together with \cite[Corollary 10.16]{EP17}.
\end{proof}

We note that the minimality assumption in Theorem~\ref{thm:ssEsimple} is satisfied in many cases, for example when the action of $G$ fixes every vertex and the graph is transitive, see \cite[Theorem~13.6]{EP17} together with note (2) below
\cite[Corollary~13.7]{EP17}.
\subsection{A simple Katsura algebra with non-Hausdorff groupoid}

In \cite{Ka08}, Katsura associates a $\cs$-algebra to a pair of square integer matrices $A,B$ and studies their properties.
These were recast as $\cs$-algebras of self-similar graphs in \cite{EP17}.
In this section we describe such a self-similar graph action which gives a groupoid which is minimal, effective, non-Hausdorff groupoid in which every compact open set is regular open.

Consider the matrices
\begin{equation}\label{eq:matrices}
A = \left[\begin{matrix}
2&1&0\\
1&2&1\\
1&1&2
\end{matrix}\right]
\hspace{.5cm} \text{and} \hspace{.5cm}
B = \left[\begin{matrix}
1&2&0\\
2&1&2\\
0&2&1
\end{matrix}\right].
\end{equation}
Let $E_A = (E_A^0, E_A^1, r, s)$ be the directed graph whose incidence matrix is $A$.
A diagram of $E_A$ is given below.
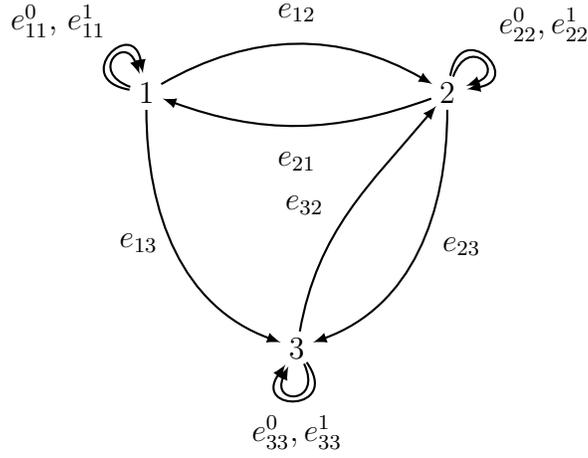
\begin{center}
	\begin{tikzpicture}
	\node at (0,0) {$1$};
	\node[vertex] (v1) at (0,0)   {$\,$}
	edge [->,>=latex,out=160,in=110,loop,thick] node[auto,swap,xshift=-0.1cm,yshift=0.1cm,pos=0.5]{} (v1)
	edge [->,>=latex,out=170,in=100,loop,thick] node[auto,xshift=0.1cm,yshift=0.1cm,pos=0.5]{$e^0_{11}$, $e^1_{11}$} (v1);
	\node at (4,0) {$2$};
	\node[vertex] (v2) at (4,0)   {$\,$}
	edge [->,>=latex,out=80,in=10,loop,thick] node[auto,swap,xshift=-0.1cm,yshift=0.1cm,pos=0.5]{} (v2)
	edge [->,>=latex,out=70,in=20,loop,thick] node[auto,xshift=0.1cm,yshift=0.1cm,pos=0.5]{$e^0_{22}, e^1_{22}$} (v2)
	edge [->,>=latex,out=200,in=340,thick] node[auto,yshift=-0.2cm]{$e_{21}$} (v1)
	edge [<-,>=latex,out=150,in=30,thick] node[auto,yshift=0.7cm]{$e_{12}$} (v1);
	\node at (2,-3.4) {$3$};
	\node[vertex] (v3) at (2,-3.4)   {$\,$}
	edge [->,>=latex,out=305,in=235,loop,thick] node[auto,swap,xshift=0cm,yshift=-0.8cm,pos=0.5]{$e^0_{33}, e^1_{33}$} (v3)
	edge [->,>=latex,out=295,in=245,loop,thick] node[auto,xshift=0.1cm,yshift=0.1cm,pos=0.5]{} (v3)
	edge [<-,>=latex,out=20,in=270,thick] node[auto,xshift=1cm,yshift=-0.2cm,pos=0.5]{$e_{23}$} (v2)
	edge [->,>=latex,out=80,in=230,thick] node[auto,xshift=-0.2cm,yshift=-0.2cm,pos=0.5]{$e_{32}$} (v2)
	edge [<-,>=latex,out=160,in=270,thick] node[auto,xshift=-0.1cm,yshift=0.4cm,pos=0.5]{$e_{13}$} (v1);
	\end{tikzpicture}
\end{center}
The matrix $B$ determines a $\Z$ action and a cocycle $\vp: \Z\times E^1_A\to \Z$, as described in \cite{EP17}. This action and cocycle are described for $1\in\Z$ as follows:
\begin{align*}
1\cdot e_{ii}^0 = e_{ii}^1, \hspace{1cm}& \vp(1, e_{ii}^0) = 0 \quad  \text{ for } i = 1, 2, 3,\\
1\cdot e_{ii}^1 = e_{ii}^0, \hspace{1cm}& \vp(1, e_{ii}^1) = 1 \quad \text{ for } i = 1,2,3,\\
1\cdot e_{12} = e_{12}, \hspace{1cm}& \vp(1, e_{12}) = 2, \\
1\cdot e_{21} = e_{21}, \hspace{1cm}& \vp(1, e_{21}) = 2,\\
1\cdot e_{32} = e_{32}, \hspace{1cm}& \vp(1, e_{32}) = 2, \\
1\cdot e_{23} = e_{23}, \hspace{1cm}& \vp(1, e_{23}) = 2, \\
1\cdot e_{13} = e_{13}, \hspace{1cm}& \vp(1, e_{13}) = 0.
\end{align*}

In what follows, we let
\begin{align}
W &= \{w\in E^*_A: r(e) = s(e)\text{ for every edge $e$ in $w$}\}\label{eq:Wdef}\\
V &= \{v\in E^*_A: r(e) \neq  s(e)\text{ for every edge $e$ in $v$}\}\label{eq:Vdef}.
\end{align}
Evidently, $\Z$ acts differently on paths in $W$ than on paths in $V$.
On paths in $W$, $\Z$ acts like a 2-odometer, while for $n\in \Z$ and $v\in V$ we have
\begin{align}
n\cdot v &= v & \vp(n, v) &= \begin{cases}0 &\text{if }e_{13}\text{ is an edge in }v\\n2^{|v|}&\text{otherwise}.\label{eq:ndotv}
\end{cases}
\end{align}
Furthermore, suppose that $n\cdot w = w$ for some $n\in \Z$ and $w\in W$.
Then we must have that $n = k2^{|w|}$ for some $k\in \Z$, and in this case
\begin{equation}
k2^{|w|}\cdot w = w \hspace{1cm}\vp(k2^{|w|}, w) = k,\label{eq:ndotw}
\end{equation}
in other words, $\vp(n, w) = 2^{-|w|}n$.
We note that this works for either positive or negative $n$, using the rule $\vp(-n, \mu) = -\vp(n, (-n)\cdot\mu)$, see \cite[Proposition 2.6]{EP16}.

\begin{lemma}\label{lem:katsMinNH}
	Let $(\Z, E_A,\vp)$ be the Katsura triple associated to the matrices \eqref{eq:matrices}, and let $\gzeA$ be the associated groupoid. Then $\gzeA$ is minimal and non-Hausdorff.
\end{lemma}
\begin{proof}
	The matrix $A$ is irreducible, and so $\gzeA$ is minimal by \cite[Theorem 18.7]{EP17}.

To show non-Hausdorff, by \eqref{eq:SSGHausdorff}, it will be enough to find an element of $\Z$ with infinitely many minimal strongly fixed paths.
We claim that the generator $1$ does the job.
Indeed, for any $k\geq 1$, if we let $\alpha^{(k)}= (e_{23}e_{32})^ke_{13}$, then
\[
1\cdot \alpha^{(k)} = \alpha^{(k)}, \hspace{1cm} \vp(1, \alpha^{(k)}) = e
\]
while if $\alpha^{(k)}_{[1,n]}$ denotes the prefix of $\alpha^{(k)}$ of length $n$, we see that
\[
1\cdot \alpha^{(k)}_{[1,n]} = \alpha^{(k)}_{[1,n]}, \hspace{1cm} \vp\left(z, \alpha^{(k)}_{[1,n]}\right) = z^{2n}.
\]
Hence $\alpha^{(k)}$ is a minimal strongly fixed word for $1$, and since they are distinct for each $k\geq 1$, $1$ has infinitely many minimal strongly fixed paths.
\end{proof}

We now show that the inverse semigroup associated to this self-similar action satisfies condition \eqref{eq:cond(S)}.
Let $x\in E_A^0$ be any vertex, and use the shorthand
\begin{equation}\label{eq:Fn}
F_n := F_{(x, n, x)},\hspace{1cm}TF_n := TF_{(x, n, x)}.
\end{equation}
\begin{lemma}\label{lem:katsTF}
	Let $(\Z, E_A,\vp)$ be the Katsura triple associated to the matrices \eqref{eq:matrices}, and let $x\in E_A^0$.
	Then keeping the notation \eqref{eq:Fn} in force, we have
	\begin{enumerate}
		\item If $\ell\in\Z$ is odd, then $F_\ell = F_1$ and $TF_\ell = TF_1$.
		\item If $\ell\in\Z$ is nonzero, even and $\ell = m2^n$ for some $n\geq 1$ and some odd $m\in\Z$, then  $F_\ell = F_{2^n}$ and $TF_\ell = TF_{2^n}$.
		\item For every $n\geq 0$ we have $F_{2^n}\subsetneq F_{2^{n+1}}$ and $TF_{2^n}\subsetneq TF_{2^{n+1}}$.
	\end{enumerate}
\end{lemma}
\begin{proof}

\begin{enumerate}
	\item Clearly, $F_1\subseteq F_\ell$ and $TF_1\subseteq TF_\ell$ for all $\ell\in\Z$.
	We prove the other containments.

Given $\xi\in E_A^\infty$ it is of one of two forms:
\[
\xi = w_1v_1w_2v_2\cdots \hspace{1cm} \text{or}\hspace{1cm} \xi = v_1w_2v_2w_3\cdots\hspace{1cm} w_i\in W, v_i\in V, i = 1, 2, \dots
\]
Suppose $\ell\in \Z$ is odd and that $\xi\in F_\ell$.
Since $\ell$ is odd, $\xi$ must be of the form $x = v_1w_2v_2w_3\cdots$.
We claim that the equation
\begin{equation}\label{eq:Katlength}
|w_j| \leq \sum_{i = 1}^{j-1}|v_i| - \sum_{i = 2}^{j-1}|w_i|
\end{equation}
must hold for every $j$ such that $v_k$ does not contain the edge $e_{13}$ for $k = 1, \dots, j$.
Suppose that we have such a $j$.
Since $\vp(\ell, v_1) = \ell 2^{|v_1|}$ (because $v_1$ does not contain $e_{13}$) and $w_2$ is fixed by multiples of $2^{|w_2|}$, we must have $|w_2|\leq |v_1|$.
So suppose that we know
\[
|w_k| \leq \sum_{i = 1}^{k-1}|v_i| - \sum_{i = 2}^{k-1}|w_i|
\]
holds for some $k = 3, \dots j-1$.
Then
\begin{align*}
\ell\cdot \xi &= \ell\cdot v_1w_2v_2w_3\cdots\\
&= v_1w_2v_2\cdots w_k\vp(\ell, v_1w_2v_2\cdots w_k)\cdot v_kw_{k+1}\cdots\\
&= v_1w_2v_2\cdots w_k(\ell2^{\sum_{i = 1}^{k-1}|v_i| - \sum_{i = 2}^{k}|w_i|})\cdot v_kw_{k+1}\cdots
\end{align*}
and we know the power of 2 is greater than or equal to 0. Hence
\begin{align*}
\ell\cdot \xi &= v_1w_2v_2\cdots w_kv_k(\ell2^{|v_k|}2^{\sum_{i = 1}^{k-1}|v_i| - \sum_{i = 2}^{k}|w_i|})\cdot w_{k+1}\cdots\\
&= v_1w_2v_2\cdots w_kv_k(\ell2^{\sum_{i = 1}^{k}|v_i| - \sum_{i = 2}^{k}|w_i|})\cdot w_{k+1}\cdots
\end{align*}
By hypothesis $\xi$ is fixed by $\ell$, and so $\ell2^{\sum_{i = 1}^{k}|v_i| - \sum_{i = 2}^{k}|w_i|}\cdot w_{k+1} = w_{k+1}$ hence $\ell2^{\sum_{i = 1}^{k}|v_i| - \sum_{i = 2}^{k}|w_i|}$ must be a multiple of $2^{|w_k|}$, i.e., $|w_k|\leq \sum_{i =
1}^{k}|v_i| - \sum_{i = 2}^{k}|w_i|$.
Hence \eqref{eq:Katlength} holds for all such $j$.
Now the same calculation as above with 1 replacing $\ell$ shows that $\xi$ is fixed by $1$.
Hence $F_1 = F_\ell$.

We now turn to the trivially fixed infinite paths.
It is straightforward to verify that a path $\xi\in F_\ell$ will be trivially fixed if and only if there is a prefix $\mu$ of $\xi$ such that $\ell\cdot \mu = \mu$ and $\vp(\ell, \mu) = 0$, and this happens if and only if the edge $e_{13}$ appears in $\xi$.
This statement does not depend on what $\ell$ is, and so we see that $TF_\ell = TF_1$.
\item Suppose $\ell$ is nonzero and even, and write $\ell = m2^n$ for $n\geq 1$ and odd $m$.
If we write $\xi = w_1v_1w_2v_2\cdots$ for $w_i\in W$, $v_i\in V$ (with the possibility that $w_1$ is the vertex $x$), then similar arguments to the above imply that $\xi\in F_\ell$ if and only if the equation
\begin{equation}\label{eq:2nfixed}
|w_j| \leq n + \sum_{i = 1}^{j-1}|v_i| - \sum_{i = 1}^{j-1}|w_i|
\end{equation}
holds for every $j$ such that $v_k$ does not contain the edge $e_{13}$ for $k = 1, \dots, j$.
This statement does not depend on $m$, so $F_\ell = F_{2^n}$. Again for similar reasons as the above, $TF_\ell = TF_{2^n}$.
\item If $\xi$ satisfies \eqref{eq:2nfixed} for every $j$ such that $v_k$ does not contain the edge $e_{13}$ for $k = 1, \dots, j$, then the same will be true if $n$ is replaced by $n+1$.
Hence $F_{2^n}\subseteq F_{2^{n+1}}$ and $TF_{2^n}\subseteq TF_{2^{n+1}}$.

To show the containments are proper, we note that at vertex 1 we have that
\[
(e_{11})^{n+1}e_{21}e_{32}e_{13}(e_{21}e_{12})^\infty
\]
is in both $TF_{2^{n+1}}\setminus TF_{2^n}$ and $F_{2^{n+1}}\setminus F_{2^n}$
and similar paths can be constructed at the other vertices.
\end{enumerate}
\end{proof}

\begin{lemma}\label{lem:katsEff}
	Let $(\Z, E_A,\vp)$ be the Katsura triple associated to the matrices \eqref{eq:matrices}, and let $\gzeA$ be the associated groupoid.
	Then $\gzeA$ is effective.
\end{lemma}
\begin{proof}
Let $\gamma\in$ (Iso$(\gzeA))^\circ$.
Then $\gamma = [(\alpha, \ell, \beta), \beta\xi]$ for $\alpha,\beta\in E_A^\infty$, $\ell\in \Z$,  and $\beta\xi$ is a fixed point for $(\alpha,n,\beta)$.
By \cite[Proposition 14.3(i)]{EP17} and the fact that our path space has no isolated points, we can assume that $|\alpha| = |\beta|$, which implies $\alpha = \beta$ and $\xi$ is a fixed point for $\ell$.

If $\xi$ is not trivially fixed by $\ell$, then the edge $e_{13}$ does not appear in $\xi$.
Hence for every prefix $\mu$ of $\xi$, we can find an element of $C(\mu)$ (say $\mu x$ where $x\in E^\infty_A$ consists only of loops at the vertex $s(\mu)$) which is not fixed by $\ell$.
Hence for every neighbourhood $U$ of $\gamma$ we can find some $[(\alpha,\ell,\beta), \mu x]\in U$ which is not in the isotropy group bundle.
This contradicts $\gamma\in$ ($\operatorname{Iso}(\gzeA))^\circ$, so $\xi$ must be trivially fixed by $n$, which implies that $\gamma\in \gzeA^{(0)}$.
Thus $\gzeA$ is effective.
\end{proof}
We can now prove that our example satisfies condition \eqref{eq:cond(S)}.

\begin{lemma}\label{lem:kats(S)}
	Let $(\Z, E_A,\vp)$ be the Katsura triple associated to the matrices \eqref{eq:matrices}, and let $\mathcal{S}_{E_A, \Z}$ be the associated inverse semigroup.
	Then $\mathcal{S}_{E_A, \Z}$ satisfies condition \eqref{eq:cond(S)}.
\end{lemma}
\begin{proof}
	We use Lemma~\ref{lem:SforSSGraphs}.
	Let $\ell_1, \ell_2, \dots, \ell_n$ be a sequence of distinct integers, and let $x\in E_A^0$.
	Without loss of generality we may assume that there exists $0\leq k\leq n$ such that the $\ell_i$ are arranged in order so that $\ell_1, \dots, \ell_k$ are odd and such that for any $j = k+1, \dots n$ we have that $\ell_j = r_j 2^{m_j}$ for some $r_j\in
\Z$ and $1\leq m_{k+1}< m_{k+2} < \dots <m_{n}$ (we take the case $k=0$ to mean that none of the $\ell_i$ are odd).
	Then
	\begin{align*}
	\bigcap_{i = 1}^{n}F_{\ell_i}\setminus TF_{\ell_i} &= \bigcap_{i = 1}^{n}F_{\ell_i}\cap TF_{\ell_i}^c\\
	&= \left(F_{\ell_1}\cap\cdots\cap F_{\ell_n}\right)\cap\left( TF_{\ell_1}^c\cap\cdots\cap TF_{\ell_n}^c\right)\\
	&= F_{\ell_1}\cap TF_{\ell_n}^c& \text{by Lemma~\ref{lem:katsTF}(3)}\\
	&= F_{\ell_1}\setminus TF_{\ell_{n}}\\
	&=\left\{\begin{array}{ll}
	F_1\setminus TF_1&\text{if }k=n\\
	F_1\setminus TF_{2^{m_n}}&\text{if }1<k<n\\
	F_{2^{m_{k+1}}}\setminus TF_{2^{m_n}}&\text{if }k=0.
	\end{array}\right.
	\end{align*}
	On the other hand, we have
	\[
		\left(\bigcup_{i=1}^n F_{\ell_i}\right)^\circ =
	\begin{cases}\mathring{F_1}&\text{if }k=n\\
	\mathring{F_{2^{m_n}}}&\text{if }1<k<n\\
	\mathring{F_{2^{m_n}}}&\text{if }k=0.
	\end{cases}
	\]
	By Lemma~\ref{lem:katsEff} $\gzeA$ is effective, so \cite[Definition 4.1]{EP16} and \cite[Theorem 4.7]{EP16} imply that $TF_\ell = \mathring{F_\ell}$ for every $\ell\in\Z$.
	
	If $k=n$, we have
	\begin{align*}
	\xi\in \bigcap_{i = 1}^{n}F_{\ell_i}\setminus TF_{\ell_i} = F_1\setminus TF_1 \implies \xi\notin TF_1 = \mathring{F_1} \implies \xi\notin 	\left(\bigcup_{i=1}^n F_{\ell_i}\right)^\circ.
	\end{align*}
	The other two cases are similar.
	Thus the conditions of Lemma~\ref{lem:SforSSGraphs} are satisfied and so we are done.
	
\end{proof}

\begin{thm}\label{thm:katsSimple}
	Let $(\Z, E_A,\vp)$ be the Katsura triple associated to the matrices \eqref{eq:matrices}, and let $\gzeA$ be the associated groupoid.
	Then
	\begin{enumerate}
		\item $\mathcal{O}_{\Z,E_A} = C^*(\gzeA)$ is simple and
		\item for any field $\KK$, the algebra $A_\KK(\gzeA)$ is simple.
	\end{enumerate}
\end{thm}
\begin{proof}
This follows from Theorem~\ref{thm:ssEsimple}, Lemma~\ref{lem:kats(S)}, Lemma~\ref{lem:katsMinNH}, and Lemma~\ref{lem:katsEff}.
\end{proof}

\subsection{Algebras of self-similar actions} \label{sec:ssg}

Let $X$ be a finite set with more than one element, let $G$ be a group, and let $X^*$ denote the set of all words in elements of $X$, including an empty word $\ew$.
Let $X^\omega$ denote the Cantor set of one-sided infinite words in $X$, with the product topology of the discrete topology on $X$.
Recall that the cylinder sets $C(\alpha) = \{\alpha x: x\in X^\omega\}$ form a clopen basis for the topology on $X^\omega$ as $\alpha$ ranges over $X^*$.

Suppose that we have a faithful length-preserving action of $G$ on $X^*$, with $(g, \alpha)\mapsto g\cdot \alpha$, such that for all $g\in G$, $x\in X$ there exists a unique element of $G$, denoted $\left.g\right|_x$, such that for all $\alpha\in X^*$
\[
g(x\alpha) = (g\cdot x)(\left.g\right|_x\cdot \alpha).
\]
In this case, the pair $(G,X)$ is called a {\em self-similar action}. The map $G\times X\to G$, $(g, x)\mapsto \left.g\right|_x$ is called the {\em restriction} and extends to $G\times X^*$ via the formula
\[
\left.g\right|_{\alpha_1\cdots \alpha_n} = \left.g\right|_{\alpha_1}\left.\right|_{\alpha_2}\cdots \left.\right|_{\alpha_n}
\]
and this restriction has the property that for $\alpha, \beta\in X^*$, we have
\[
g(\alpha\beta) = (g\cdot \alpha)(\left.g\right|_{\alpha}\cdot \beta).
\]
The action of $G$ on $X^*$ extends to an action of $G$ on $X^\omega$ given by
\[
g\cdot (x_1x_2x_3\dots) = (g\cdot x_1)(\left.g\right|_{x_1}\cdot x_2)(\left.g\right|_{x_1x_2}\cdot x_3)\cdots .
\]
Notice that if $R_{\vert X\vert}$ denotes the graph with a single vertex and $\vert X\vert$ edges, and $\varphi(g,\alpha):=\left.g\right|_{\alpha}$, then the self-similar group $(G, X)$ is equivalent to the self-similar graph triple $(G , R_{\vert X\vert},
\varphi)$.

In \cite{Nek09}, Nekrashevych associates a C*-algebra to $(G, X)$, denoted $\mathcal{O}_{G, X}$, which is the universal C*-algebra generated by a set of isometries $\{s_x\}_{x\in X}$ and a unitary representation $\{u_g\}_{g\in G}$ satisfying
\begin{enumerate}
	\item[(i)]$s_x^*s_y = 0$ if $x\neq y$,
	\item[(ii)]$\sum_{x\in X}s_xs_x^* = 1$ and
	\item[(iii)]$u_gs_x = s_{g\cdot x}u_{\left.g\right|_x}$.
\end{enumerate}
The Nekrashevych $\cs$-algebra $\mathcal{O}_{G, X}$ turns out to be the self-similar graph $\cs$-algebra $\mathcal{O}_{G, R_{\vert X\vert}}$ associated to the triple $(G , R_{\vert X\vert}, \varphi)$. Hence, the results in previous subsection apply here.
Nevertheless, we will recall the specific construction of the groupoid, as it will be useful for understanding the example in next subsection.

As before, one can express $\mathcal{O}_{G,X}$ as the tight C*-algebra of an inverse semigroup. Let
\[
S_{G, X} = \{ (\alpha, g, \beta): \alpha,\beta\in X^*, g\in G\}\cup\{0\}.
\]
This set becomes an inverse semigroup when given the operation
\[
(\alpha, g, \beta)(\gamma, h, \nu) = \begin{cases}(\alpha (g\cdot\varepsilon), \left.g\right|_{\varepsilon}h, \nu) &\text{if }\gamma = \beta\varepsilon,\\ (\alpha, g(\left.h^{-1}\right|_{\varepsilon})^{-1}, \nu (h^{-1}\cdot\varepsilon)) & \text{if } \beta =
\gamma\varepsilon,\\ 0 &\text{otherwise}\end{cases}
\]
with
\[
(\alpha, g, \beta)^* = (\beta, g^{-1}, \alpha).
\]
The set of idempotents is given by
\[
E(\sgx) = \{(\alpha, 1_G, \alpha): \alpha \in X^*\}\cup\{0\}.
\]
The tight spectrum of $E(\sgx)$ is homeomorphic to $X^\omega$, and the standard action of $\sgx$ on its tight spectrum is realized as follows: for $\alpha, \beta\in X^*$ and $g\in G$, let
\[
\theta_{(\alpha, g, \beta)}: C(\beta) \to C(\alpha)
\]
\[
\theta_{(\alpha, g, \beta)}(\beta w) = \alpha(g\cdot w)
\]
for every $w\in X^\omega$.

We use the notation $\ggx:= \gt(\sgx)$. Then $\ggx$ is ample and minimal, and since the action of $G$ on $X^*$ is faithful, then $\ggx$ is effective, see \cite[Section 17]{EP17}.
The C*-algebra is isomorphic to $\cs(\ggx)$, see \cite[Example~3.3]{EP17} and \cite[Corollary~6.4]{EP17}.
We keep the notation in \eqref{eq:FWdef}, \eqref{eq:SFWdef}, \eqref{eq:MSFWdef} in force (though we call their elements {\em words} rather than paths) and note that \eqref{eq:SSGHausdorff} still characterizes when $\ggx$ is Hausdorff.

We record the translation of Lemma~\ref{lem:inclosure} to this context.

\begin{lemma}\label{lem:ssgcalc}
	Let $(G, X)$ be a self-similar action, let $(\alpha, g, \beta)\in \sgx$, let $\eta\in X^*$, and let $U = \Theta((\alpha, g, \beta), C(\beta\eta))$. Let $z = [(\gamma, h, \delta), \delta w]$ for some $\gamma, \delta\in X^*$, $w\in X^\omega$, and $h\in G$
with $|\delta|\geq |\beta\eta|$. Then
	
	\begin{enumerate}
		\item $z\in \overline U$ if and only if the following conditions hold:
		\begin{enumerate}
			\item $\delta = \beta\eta\eta'$ and $\gamma = \alpha(g\cdot(\eta\eta'))$ for some $\eta'\in X^*$, and
			\item every prefix of $w$ can be extended to a strongly fixed word for $h^{-1}\left. g\right|_{\eta\eta'}$.
		\end{enumerate}
		\item In the case of (1), $z\notin U$ if and only if no prefix of $w$ is strongly fixed by $h^{-1}\left. g\right|_{\eta\eta'}$.
	\end{enumerate}

\end{lemma}
\begin{proof}
	This is direct from Lemma~\ref{lem:inclosure} and is left to the reader.
\end{proof}

We note that Lemma~\ref{lem:ssgcalc}(1) implies that if $[(\gamma, h, \delta), \delta w]\in \overline U$, then $w\in X^\omega$ is fixed by $h^{-1}\left. g\right|_{\eta\eta'}$ since every prefix of $w$ extends to a strongly fixed word for $h^{-1}\left.
g\right|_{\eta\eta'}$.

\begin{lemma}\label{lem:basisRO}
	Let $(G, X)$ be a self-similar action.
	If $(G, X)$ is faithful, then for all $(\alpha, g, \beta)\in \sgx$ and all $\eta\in X^*$, the compact open bisection $\Theta((\alpha, g, \beta), C(\beta\eta))$ is regular open.
\end{lemma}
\begin{proof}
	If $(G,X)$ is faithful, $\ggx$ is effective. Hence this follows from Lemma~\ref{lem:Effective=> basic reg}.
\end{proof}

As mentioned before, Lemma~\ref{lem:basisRO} is not enough to allow us to conclude that every compact open set is regular open, see Example~\ref{sec:grig}.
The following condition is enough to guarantee that all compact open sets are regular open.

\begin{definition}\label{def:omegafaithful}
	We say that a self-similar action $(G, X)$ is {\em $\omega$-faithful} if for every $F\fs G$ and every $x\in X^\omega$ such that every prefix of $x$ is in $\cap_{f\in F}FW_f\setminus SFW_f$, there exists $n\in \N$ such that for every prefix $\mu$ of $x$
with $|\mu|\geq n$ there exists a word $\xi$ such that $\left.f\right|_\mu\cdot \xi \neq \xi$ for all $f\in F$.
\end{definition}

\begin{lemma}\label{lem:condF}Let $(G, X)$ be an $\omega$-faithful self-similar action.
Then $\sgx$ satisfies condition \eqref{eq:cond(S)} and hence every compact open subset of $\ggx$ is regular open.
\end{lemma}
\begin{proof}
	We use Lemma~\ref{lem:SforSSGraphs}.
	Suppose that we have $F\fs G\setminus \{1_G\}$ such that $f\cdot x = x$ for all $f\in F$ but that $x$ is not trivially fixed by any $f\in F$.
	Then every prefix of $x$ must be in $\cap FW_g\setminus SFW_g$.
	Let $n$ be the element of $\N$ guaranteed to exist by the definition of $\omega$-faithful, let $\mu$ be a prefix of $x$ longer than $n$, and consider the corresponding cylinder set $C(\mu)$.
	We show that $C(\mu)$ contains an element not fixed by any $f\in F$.
	
	Let $\xi$ be a word such that $\left.f\right|_\mu\cdot \xi \neq \xi$ for all $f\in F$.
	Then for any $y\in X^\omega$, $\mu\xi y\in C(\mu)$ is not fixed by any element of $F$, so $x$ is not in the interior of $\cup F_f$.
	Hence by Lemma~\ref{lem:SforSSGraphs}, $\sgx$ satisfies \eqref{eq:cond(S)}.

	\end{proof}
We can now state the culmination of our results in the case of self-similar actions.
\begin{thm}\label{thm:ssgsimple}
	Let $(G, X)$ be an $\omega$-faithful self-similar action. Then
	\begin{enumerate}
		\item $\cs_r(\ggx)$ is simple.
		\item If $G$ is amenable, then $\ogx$ is simple.
		\item For any field $\KK$, the algebra $A_\KK(\ggx)$ is simple.
	\end{enumerate}
\end{thm}
\begin{proof}
	Faithfulness implies that $\ggx$ is effective.
	Regardless of what $G$ is or its action, $\ggx$ is always minimal (see for example \cite[Lemma 4.2]{Sta15}).
	By Lemma~\ref{lem:condF}, every compact open subset of $\ggx$ is regular open.
	Now (3) follows from Theorem~\ref{thm:simple}, and (1) and (2) follow from Theorem~\ref{thm:c*simple} together with \cite[Corollary 10.16]{EP17}.
\end{proof}

\subsection{The Grigorchuk group}\label{sec:grig}
\label{grig}

Let $X = \{0,1\}$, and let $a, b, c, d$ be the length-preserving bijections of $X^*$ determined by the formulas

$$
\begin{array}{lll}
a\cdot(0w) = 1w && c\cdot (0w) = 0(a\cdot w)\\
a\cdot(1w) = 0w && c\cdot (1w) = 1(d\cdot w)\\
\\
b\cdot(0w) = 0(a\cdot w) && d\cdot (0w) = 0w\\
b\cdot(1w) = 1(c\cdot w) && d\cdot (1w) = 1(b\cdot w)
\end{array}
$$
for every $w\in X^*$. The group $G$ generated by the set of bijections $\{a,b,c,d\}$ is called the {\em Grigorchuk group}, and was the first example of an amenable group with intermediate word growth \cite{Gr80, Gr84}.
It is worth noting that we have the relations $a^2 = b^2 = c^2 = d^2 = e$ (here we write the identity element of $G$ as $e$), $bc = d = cb$, $db=c = bd$, and $cd= b = dc$.
We also have that $(G, X)$ is a faithful self-similar action with restrictions given by
$$
\begin{array}{lll}
\left.a\right|_0 = e&& \left.c\right|_0 = a\\
\left.a\right|_1 = e && \left.c\right|_1 = d\\
\\
\left.b\right|_0 = a && \left.d\right|_0 = e\\
\left.b\right|_1 = c && \left.d\right|_1 = b.
\end{array}
$$

Each of the elements $b, c,$ and $d$ have infinitely many minimally fixed words. Indeed, some short calculations
\footnote{These calculations were aided by following the presentation of the Grigorchuk group and its nucleus given in \cite[Section 2.4]{LRRW14}.}
show that for all $n\in \N$,

\begin{eqnarray}
b\cdot(1^{3n+2}0) = 1^{3n+2}0 && \left.b\right|_{1^{3n+2}0} = e\nonumber\\
c\cdot(1^{3n+1}0) = 1^{3n+1}0 && \left.c\right|_{1^{3n+1}0} = e\label{eq:restrictions}\\
d\cdot(1^{3n}0) = 1^{3n}0 && \left.d\right|_{1^{3n}0} = e\nonumber
\end{eqnarray}
and that none of the prefixes of the words in question are strongly fixed. Hence $\ggx$ is not Hausdorff.

In what follows, we will make use of the fact that $(G,X)$ is {\em contracting with nucleus} $\{e,a,b,c,d\}$, which means that for any $g\in G$ there exists $n\geq 0$ such that $\left.g\right|_v\in \{e, a, b, c, d\}$ for all $v\in X^*$ with $|v|\geq n$, see
\cite[Proposition 2.7]{LRRW14}.

The remainder of this paper is devoted to proving the following theorem.

\begin{thm}\label{thm:grigsimple}
	Let $G$ be the Grigorchuk group, let $(G,X)$ be its self-similar action, and let $\ggx$ be the associated groupoid. Then
	\begin{enumerate}
		\item\label{it:grigthm1} For any field $\KK$ of characteristic zero, $A_\KK(\ggx)$ is simple, and
		\item\label{it:grigthm2} $\ogx$ is simple.
	\end{enumerate}
\end{thm}

Interestingly, simplicity can fail when $\KK$ has nonzero characteristic, see Corollary~\ref{cor:grigZ2notsimple}.

We begin by observing that $\ggx$ has compact open subsets which are not regular open.
\begin{lemma}\label{lem:compactopennotregularopen}
	Let
	\[
	U = \Theta((\ew,b,\ew), X^\omega) \cup \Theta((\ew, c, \ew), X^\omega) \cup\Theta(\ew, d,\ew)), X^\omega),
	\]
	then $U$ is a compact open set which is not regular open.
\end{lemma}
\begin{proof}It is clear that $U$ is compact and open because it is a union of compact open sets.

Let $z_e = [(\ew,e, \ew), 1^\infty]$.
Every prefix of $1^\infty$ can be extended to a strongly fixed word for each of $b, c, d$ by \eqref{eq:restrictions}, but no prefix of $1^\infty$ is strongly fixed by any of these elements. Thus $z_e\in \overline U\setminus U$.

We will find a neighbourhood $V$ of $z_e$ such that $V\setminus \{z_e\}\subseteq U$, which will show that $z_e$ is an interior point of $\overline U$.
We take $V = \ggx^{(0)}$, the entire unit space of $\ggx$.
Every point in $\ggx^{(0)}\setminus \{z_e\}$ is of the form $ [(\ew,e, \ew), 1^n0x]$ for some $x\in X^\omega$ and $n\geq 0$.

If $n \equiv 0 \mod 3$ then $n = 3m$ for some $m\geq 0$ and
\[
(\ew, d, \ew)(1^{3m}0x, e, 1^{3m}0x) = (1^{3m}0x, e, 1^{3m}0x) = (\ew,e, \ew)(1^{3m}0x, e, 1^{3m}0x)
\]
and so
\[
[(\ew,e, \ew), 1^n0x] = [(\ew, d, \ew), 1^n0x] \in U.
\]
If $n \equiv 1 \mod 3$ then $n = 3m+1$ for some $m\geq 0$ and
\[
(\ew, c, \ew)(1^{3m+1}0x, e, 1^{3m+1}0x) = (1^{3m+1}0x, e, 1^{3m+1}0x) = (\ew,e, \ew)(1^{3m+1}0x, e, 1^{3m+1}0x)
\]
and so again
\[
[(\ew,e, \ew), 1^n0x] = [(\ew, c, \ew), 1^n0x] \in U.
\]
If $n \equiv 2 \mod 3$ then $n = 3m+2$ for some $m\geq 0$ and
\[
(\ew, b, \ew)(1^{3m+2}0x, e, 1^{3m+2}0x) = (1^{3m+2}0x, e, 1^{3m+2}0x) = (\ew,e, \ew)(1^{3m+2}0x, e, 1^{3+2m}0x)
\]
and so once again we have
\[
[(\ew,e, \ew), 1^n0x] = [(\ew, b, \ew), 1^n0x] \in U.
\]
Hence $\ggx^{(0)}\setminus \{z_e\}\subseteq \overline U$, and since $z\in \overline U\setminus U$, $z_e$ is an interior point of $\overline U$.
This proves that $U$ is not regular open.
\end{proof}

Even though $\ggx$ admits a compact open set which is not regular open, we can still prove that $A_\KK(\ggx)$ is simple for any characteristic zero field $\KK$.
We do this over a sequence of lemmas.
First, for $g\in G$ and $m\in \N$,  we use the notation
\[
U_{g, m} := \Theta((\varnothing, g, \varnothing), C(1^m)).
\]
\[
U_{g, m}' := \Theta((\varnothing, g, \varnothing), C(1^m0)).
\]
In addition, following the notation set in the proof of Lemma~\ref{lem:compactopennotregularopen} we set
\begin{align}
z_e := [(\varnothing, e, \varnothing), 1^\infty], && z_c := [(\varnothing, c, \varnothing), 1^\infty],\label{eq:zebc}\\
z_b := [(\varnothing, b, \varnothing), 1^\infty], && z_d := [(\varnothing, d, \varnothing), 1^\infty].\nonumber
\end{align}
\begin{lemma}\label{lem:GrigInt}
	For all $m\in \N$, we have
	\begin{align*}
		U_{b,m}\cap U_{e, m} &= \bigcup_{\substack{n\in\N\\3n+2\geq m}} U_{b, 3n+2}' =\bigcup_{\substack{n\in\N\\3n+2\geq m}} U_{e, 3n+2}'\\
		U_{c,m}\cap U_{e, m} &= \bigcup_{\substack{n\in\N\\3n+1\geq m}}  U_{c, 3n+1}' =\bigcup_{\substack{n\in\N\\3n+1\geq m}} U_{e, 3n+1}'\\
		U_{d,m}\cap U_{e, m} &=  \bigcup_{\substack{n\in\N\\3n\geq m}}U_{d, 3n}' =\bigcup_{\substack{n\in\N\\3n\geq m}} U_{e, 3n}'\\
		U_{c,m}\cap U_{d, m} &=  \bigcup_{\substack{n\in\N\\3n+2\geq m}}U_{c, 3n+2}' =\bigcup_{\substack{n\in\N\\3n+2\geq m}}U_{d, 3n+2}'\\
		U_{b,m}\cap U_{d, m} &=  \bigcup_{\substack{n\in\N\\3n+1\geq m}}U_{b, 3n+1}' =\bigcup_{\substack{n\in\N\\3n+1\geq m}} U_{d, 3n+1}'\\
		U_{c,m}\cap U_{b, m} &=  \bigcup_{\substack{n\in\N\\3n\geq m}}U_{c, 3n}' =\bigcup_{\substack{n\in\N\\3n\geq m}} U_{b, 3n}'.
	\end{align*}
\end{lemma}
\begin{proof}
	We prove the first and fourth lines --- the rest are similar.
	
	Suppose that $z\in U_{b,m}\cap U_{e, m}$, so that there exists $w\in X^\omega$ such that
	\[
	z = [(\varnothing, b, \varnothing), 1^mw] = [(\varnothing, e, \varnothing), 1^mw].
	\]
	Hence there exists $\alpha\in X^*$ such that $w = \alpha v$ and
	\begin{align*}
		(\varnothing, b, \varnothing)(1^m\alpha, e, 1^m\alpha)&=(\varnothing, e, \varnothing)(1^m\alpha, e, 1^m\alpha)\\
		(b\cdot(1^m\alpha), \left.b\right|_{1^m\alpha}, 1^m\alpha)&=(1^m\alpha, e, 1^m\alpha)
	\end{align*}
	which implies that $1^m\alpha$ is strongly fixed by $b$, and so $1^m\alpha = 1^{3n+2}$ for some $n\geq 0$.
	Furthermore, any $w \in C(\alpha)$ of this form will produce such a $z$ in the intersection, so
	\[
	U_{b,m}\cap U_{e, m} = \bigcup_{\substack{n\in\N\\3n+2\geq m}} U_{b, 3n+2}' =\bigcup_{\substack{n\in\N\\3n+2\geq m}} U_{e, 3n+2}'.
	\]
	Now suppose that $z \in U_{c,m}\cap U_{d, m}$, so that there exists $w\in X^\omega$ such that
	\[
	z = [(\varnothing, c, \varnothing), 1^mw] = [(\varnothing, d, \varnothing), 1^mw].
	\]
	Hence there exists $\alpha\in X^*$ such that $w = \alpha v$ and
	\begin{align*}
		(\varnothing, c, \varnothing)(1^m\alpha, e, 1^m\alpha)&=(\varnothing, d, \varnothing)(1^m\alpha, e, 1^m\alpha)\\
		(c\cdot(1^m \alpha), \left.c\right|_{1^m\alpha}, 1^m\alpha)&=(d\cdot(1^m \alpha), \left.d\right|_{1^m\alpha}, 1^m\alpha)
	\end{align*}
	implying that $c\cdot (1^m\alpha) = d \cdot(1^m \alpha)$ and $\left.c\right|_{1^m\alpha} = \left.d\right|_{1^m\alpha}$, and hence $1^m\alpha$ is strongly fixed by $cd^{-1} = cd = b$. So as before
	\[
	U_{c,m}\cap U_{d, m} =  \bigcup_{\substack{n\in\N\\3n+2\geq m}}U_{c, 3n+2}' =\bigcup_{\substack{n\in\N\\3n+2\geq m}} U_{d, 3n+2}'.
	\]
\end{proof}
\begin{lemma}\label{lem:LCzero}
	Let $m\in \N$, let $\KK$ be a field of characteristic zero, and let
	\begin{equation}\label{eq:singularLC}
	f = \sum_{g\in \{e, b, c, d\}}c_g 1_{U_{g, m}}
	\end{equation}
	for some $c_g\in \KK$. Then supp$(f)$ has empty interior if and only if $f$ is identically zero.
\end{lemma}
\begin{proof}
	We can write the support of such an $f$ as the disjoint union of sets of the form given in \eqref{eq:support} for $F_1, F_2$ a partition of $\{U_{g, m}\}_{g = e, b, c, d}$.
	From Lemma~\ref{lem:GrigInt} it is straightforward to see that the intersection of any three of these four sets is empty.
	By looking at the first three lines of Lemma~\ref{lem:GrigInt}, one can see that $U_{b, m}\cup U_{c, m}\cup U_{d,m}$ will contain every point of the form $[(\varnothing, e, \varnothing), 1^mw]$ where $w\in X^\omega$ contains at least one $0$ --- and so it
includes every point in $U_{e, m}$ except for $z_e = [(\varnothing, e, \varnothing), 1^\infty]$. Hence
	\[
	U_{e, m}\setminus(U_{b, m}\cup U_{c, m}\cup U_{d,m}) = \{z_e\}.
	\]
	Similarly, we have
	\begin{align*}
		U_{b, m}\setminus(U_{c, m}\cup U_{d, m}\cup U_{e,m}) & =  \{z_b\}\\
		U_{c, m}\setminus(U_{d, m}\cup U_{e, m}\cup U_{b,m}) & =  \{z_c\}\\
		U_{d, m}\setminus(U_{e, m}\cup U_{b, m}\cup U_{c,m}) & =  \{z_d\}
	\end{align*}
	all of which have empty interior.
	Hence there are six sets of the form \eqref{eq:support} which have nonempty interior, and they are listed in Lemma~\ref{lem:GrigInt}.
	Then supp$(f)$ has empty interior if and only if $f$ takes the value zero on these sets.
	This leads to the following six equations
	\begin{align}
		c_e + c_b = 0 && c_c + c_d = 0\nonumber\\
		c_e + c_c = 0 && c_b + c_d = 0\label{eq:homlinearequations}\\
		c_e + c_d = 0 && c_c + c_b = 0\nonumber
	\end{align}
	and since $\KK$ has characteristic zero, this has the unique solution $c_e = c_b = c_c = c_d = 0$.
	Hence supp$(f)$ has empty interior if and only if $f$ is identically zero.
\end{proof}

Before moving on, we note an interesting corollary of the above proof. We learned, subsequent to submission of this paper, that this result
is also proved in \cite[Example~4.5]{Nek16}.

\begin{cor}\label{cor:grigZ2notsimple}
	Let $G$ be the Grigorchuk group, let $(G,X)$ be its self-similar action, and let $\ggx$ be the associated groupoid.
	Then $\s_{\Z_2}(\ggx)$ is nonzero, and hence the Steinberg algebra $A_{\Z_{2}}(\ggx)$ is not simple.
\end{cor}
\begin{proof}
	Consider $f:= \sum_{g\in \{e,b,c,d\}} 1_{U_{g,1}}$, that is, take $c_g = 1\in\Z_2$ for all $g\in \{e,b,c,d\}$ in \eqref{eq:singularLC}.
	Then the equations \eqref{eq:homlinearequations} are satisfied, and so supp$(f)$ has empty interior.
	From the proof of Lemma~\ref{lem:LCzero}, it is clear that in fact $f = 1_{\{z_e, z_b, z_c, z_d\}}$ is the characteristic function of a set with empty interior.
	Hence $f$ is singular.
\end{proof}

We now prove that if the closure of a basic compact open bisection contains one of the points in \eqref{eq:zebc}, it contains them all.
\begin{lemma}\label{zebcd}
	Let $D = \Theta((\alpha, g, \beta), C(\beta\eta))$. Then either
	\[
	\{z_e, z_b, z_c, z_d\}\subseteq \overline{D}\hspace{1cm}\text{or}\hspace{1cm}\{z_e, z_b, z_c, z_d\}\cap \overline{D} = \emptyset
	\]
\end{lemma}
\begin{proof}
	Since $(G,X)$ is contracting with nucleus $\{e,a,b,c,d\}$, we can find $m$ so that $\left.g\right|_v\in \{e,a,b,c,d\}$ for all $v\in X^*$ such that $|v|\geq m$, and also such that $|\beta\eta|+m$ is divisible by 3.
	Take $h\in \{e, b, c, d\}$.
	By Lemma~\ref{lem:ssgcalc}, $z_h = [(1^{|\beta\eta|+m}, h, 1^{|\beta\eta|+m}), 1^\infty]\in \overline D$ implies that
	\[
	\beta\eta = 1^{|\beta\eta|}, \hspace{1cm}\alpha = 1^{|\alpha|},\hspace{1cm}g\cdot(1^{|\eta|+m})= 1^{|\eta|+m},
	\]
	and every prefix of $1^\infty$ can be extended to a strongly fixed word for $h^{-1}\left.g\right|_{1^{|\eta|+m}}\in\{e,a,b,c,d\}$.
	Hence, $h^{-1}\left.g\right|_{1^{|\eta|+m}}$ is either $e, b, c,$ or $d$.
	
	We can write $D = \Theta((1^{|\beta|}, g, 1^{|\beta|}), C(1^{|\beta\eta|})) = \Theta((1^{|\beta\eta|},\left.g\right|_{1^{|\eta|}}, 1^{|\beta\eta|}), C(1^{|\beta\eta|}))$ and we denote $\left.g\right|_{1^{|\eta|+m}} = k$.
	
	The compact open bisection $B:= \Theta((1^{|\beta\eta|+m}, k, 1^{|\beta\eta|+m}), C(1^{|\beta\eta|+m}))$ is contained in $D$.
	Take $f\in \{e, b, c, d\}$, and find $n$ such that $3n\geq |\beta\eta|+ m$.
	Then by Lemma~\ref{lem:ssgcalc}, the element $z_f = [(\varnothing, f, \varnothing), 1^\infty] = [(1^{3n}, f, 1^{3n}), 1^\infty]$ is in $\overline{B}$, because no matter what $f$ is every prefix of $1^\infty$ can be extended to a strongly fixed word for
$f^{-1}\left.k\right|_{1^{3n-|\beta\eta|-m}} \in \{e, b, c, d\}$.
	Hence $z_f\in \overline B\subseteq \overline D$.
	Since $f$ was arbitrary, we have $\{z_e, z_b, z_c, z_d\}\subseteq \overline B\subseteq \overline D$.
\end{proof}

We note that the proof of Lemma~\ref{zebcd} shows that for the closure of a basic compact bisection to contain $z_e$, it must be of the form $D = \Theta((1^k, g, 1^k), C(1^k))$ for some $k\geq 0$, where $\left.g\right|_{1^m}\in \{e, b, c, d\}$ for some $m
\geq 0$.

In the next lemma, we use the notation
\begin{equation}
U_m := U_{e, m}\cup U_{b, m}\cup U_{c, m}\cup U_{d, m}.
\end{equation}
\begin{lemma}\label{lem:Ugm}
	Suppose $D$ is a basic compact bisection whose closure contains $z_e$.
	Then there exists $h\in\{e, b, c, d\}$ and  $n\geq 0$ such that
	\[
	D\cap U_n = D\cap U_{h,n} = U_{h,n}.
	\]
\end{lemma}
\begin{proof}
	As noted above the lemma, there exists $k\geq 0$ such that $D = \Theta((1^k, g, 1^k), C(1^k))$ where $\left.g\right|_{1^m}\in \{e, b, c, d\}$ for some $m \geq 0$.
	Then as in the proof of Lemma~\ref{zebcd}, we have $B = \Theta((1^{k+m}, \left.g\right|_{1^m}, 1^{k+m}), C(1^{k+m}))$ is contained in $D$.
	Then there exists one and only one $h\in \{e, b, c, d\}$ such that $\left.h\right|_{1^{k+m}} = \left.g\right|_{1^m}$, and it follows from Lemma~\ref{lem:ssgcalc} that $U_{h, k+m}\subseteq B\subseteq D$.
	
	Now suppose $f\in\{e, b, c, d\}\setminus \{h\}$, and consider $D\cap U_{f, k+m}$.
	If $z = [(\varnothing, f, \varnothing), 1^{k+m}w]\in U_{f, k+m}\cap D$, we have $z = [(1^{k}, g, 1^k), 1^{k+m}w] = [(1^{k+m}, \left.g\right|_{1^m}, 1^{k+m}), 1^{k+m}w]$.
	So there exists $\alpha\in X^*$ with $w = \alpha y$ for some $y\in X^\omega$ such that
	\[
	\left.f\right|_{1^{k+m}}\cdot\alpha = \left.g\right|_{1^m}\cdot\alpha,\hspace{1cm} \left.g\right|_{1^m}\left.\right|_\alpha = \left.f\right|_{1^{k+m}}\left.\right|_\alpha.
	\]
	This implies that $\alpha$ is strongly fixed by $\left(\left.f\right|_{1^{k+m}}\right)^{-1}\left.g\right|_{1^m}$.
	Since $f\neq h$, this group element is not the identity.
	We have
	\begin{align*}
		[(\varnothing, f, \varnothing), 1^{k+m}\alpha y] &=  [(\varnothing, f, \varnothing)(1^{k+m}\alpha, e,1^{k+m}\alpha), 1^{k+m}\alpha y]\\
		&= [(1^{k+m}\left.f\right|_{1^{k+m}}\cdot\alpha, \left.f\right|_{1^{k+m}}\left.\right|_\alpha, 1^{k+m}\alpha), 1^{k+m}\alpha y]\\
		&= [(1^{k+m}\left.g\right|_{1^{m}}\cdot\alpha, \left.g\right|_{1^{m}}\left.\right|_\alpha, 1^{k+m}\alpha), 1^{k+m}\alpha y]\\
		&= [(1^{k+m}\left.h\right|_{1^{k+m}}\cdot\alpha, \left.h\right|_{1^{k+m}}\left.\right|_\alpha, 1^{k+m}\alpha), 1^{k+m}\alpha y]\\
		&= [(\varnothing, h, \varnothing), 1^{k+m}\alpha y]\in U_{h, k+m}.
	\end{align*}
	So $ U_{f, k+m}\cap D\subseteq U_{h, k+m}$, and paired with the above we conclude that $D\cap U_{k+n} = U_{h, k+m}$.
\end{proof}
We now show that if a function is nonzero at the point $[(\varnothing, e, \varnothing), 1^\infty]$, then its support has nonempty interior.
\begin{lemma}\label{lem:zeinterior}
	Let $\KK$ be a field of characteristic zero, and let $f\in A_\KK(\ggx)$.
	If $f(z_e)\neq 0$, then supp$(f)$ has nonempty interior.
\end{lemma}
\begin{proof}
	Let $f = \sum_{D\in F}c_D1_{D}$ for some finite set $F$ of compact open bisections and $c_D\in \KK$ for $D\in F$.
	By Lemma~\ref{lem:basisISG} and \cite[Lemma~4.14]{St}, we may assume each element of $F$ is of the form $\Theta((\alpha, g,\beta), C(\beta\eta))$.
	
	Let $F_e = \{D\in F: z_e\in \overline D\}$ and write
	\[
	f = \sum_{D\in F_e}c_D1_D + \sum_{B\notin F_e}c_B1_B.
	\]
	If $B\notin F_e$, for each $g\in \{e, b, c, d\}$ there exists $m_g\geq 0$ such that $U_{g, m}$ is disjoint from $B$ (because such sets form a neighbourhood basis of $z_g$ and $z_g$ is not in the closure of $B$).
	Then if we let $m_B = \max_{g\in\{e, b, c, d\}}\{m_g\}$, $U_{m_B}$ is disjoint from $B$. Likewise, for $D\in F_e$, Lemma~\ref{lem:Ugm} tells us we can find $m_D\geq 0$ such that $U_{m_D}\cap D = U_{m_D}$.
	
	Define $m:= \max_{D\in F}\{m_D\}$. Then we can write
	\begin{align*}
		\left.f\right|_{U_{m}} &= \left.\left(\sum_{D\in F_e}c_D1_D + \sum_{B\notin F_e}c_B1_B\right)\right|_{U_m}\\
		&= \left.\left(\sum_{D\in F_e}c_D1_D \right)\right|_{U_m}\\
		&= \left(\sum_{D\in F_e}c_D\left.\left(1_D\right)\right|_{U_{m}} \right)\\
		&= \sum_{g = e, b, c, d} \left(\sum_{\substack{ D\in F_e\\D\cap U_m = U_{g, m}}}c_D\right)1_{U_{g,m}}
	\end{align*}
	Note that while we would not expect $A_\KK(\ggx)$ to be closed under function restriction, in this case it happens that 	$\left.f\right|_{U_{m}}\in A_\KK(\ggx)$.
	Since $z_e\in U_m$ and $f(z_e)\neq 0$, this function is not identically zero.
	Hence by Lemma~\ref{lem:LCzero}, the support of this function has nonempty interior.
	The support of this function is contained in the support of $f$, so we conclude that supp$(f)$ has nonempty interior.
\end{proof}

\begin{lemma}\label{lem:grigsingular}
	Suppose that $\KK$ is a field of characteristic zero, let $f\in A_\KK(\ggx)$, and suppose $f\in \s_\KK(\ggx)$ (that is, supp$(f)$ has empty interior). Then $f$ is identically zero.
\end{lemma}
\begin{proof}
	Suppose that $\gamma \in$ supp$(f)$.
	Since supp$(f)$ has empty interior, by Lemma~\ref{lem:SupportElementBoundary} there must be a compact open bisection $D$ such that $\gamma\in \overline D \setminus D$, and by \cite[Lemma~4.14]{St} we can assume that $D$ is of the form $\Theta((\alpha, g,
\beta), C(\beta\eta))$.
	By Lemma~\ref{lem:ssgcalc} this implies that there exists $s\in \sgx$ and $\mu\in X^*$ such that $\gamma = [s, \mu1^\infty]$.
	If we let
	\[
	\xi = [(\mu, e, 1^{|\mu|}), 1^\infty]
	\]
	then a short calculation shows that $\xi^{-1}\gamma^{-1}\gamma\xi = z_e$.
	Find compact open bisections $B$ and $C$ with $(\gamma\xi)^{-1}\in B$, $\xi\in C$. Then we have
	\begin{align*}
		1_B * f* 1_C (z_e) &= 1_B* f* 1_C (\xi^{-1}\gamma^{-1}\gamma\xi)\\
		&= f* 1_C(\gamma\xi)\\
		&= f(\gamma) \neq 0.
	\end{align*}
	The singular elements form an ideal by Proposition~\ref{prop:singularideal}, and so $1_B * f * 1_C$ is singular.
	But by Lemma~\ref{lem:zeinterior}, singular elements must be zero at $z_e$, a contradiction.
	Hence no such $\gamma$ exists, which implies that $f$ is identically zero.
\end{proof}
We note that Lemmas~\ref{lem:compactopennotregularopen} and \ref{lem:grigsingular} combine to show that in Lemma~\ref{lem:key}, \eqref{lemit3:key} is strictly weaker than \eqref{lemit1:key} and \eqref{lemit2:key}.
\begin{proof}[Proof of Theorem~\ref{thm:grigsimple}\eqref{it:grigthm1}]
	This follows from Theorem~\ref{thm:simple} and Lemma~\ref{lem:grigsingular}.
\end{proof}

We now turn to the $\cs$-algebra of $\ggx$.
By Theorem~\ref{thm:c*simple}, we need to prove that for every nonzero $a\in \cs_r(\ggx)$ we have that $\supp(j(a))$ has nonempty interior.
For a given $a\in \cs_r(\ggx)$, by density of $A_\CC(\ggx)$ we can find a sequence $(f_n)$ in $A_\CC(\ggx)$ converging to $a$, and so
\[
\|j(a) - f_n\|_\infty = \|j(a - f_n)\|_\infty \leq \|a-f_n\| \to 0.
\]
Hence $j(a)$ is a uniform limit of elements of $A_\CC(\ggx)$.

We proceed as we did in the Steinberg algebra case --- prove our result at the point $z_e$ and then translate it to an arbitrary point.

	\begin{lemma}\label{lem:LCboundedbelow}
		Let $m\in \N$, and let
		\[
		f = \sum_{g\in \{e, b, c, d\}}c_g 1_{U_{g, m}} \in A_\CC(\ggx)
		\]
		for some $c_g\in \CC$.
		If $f(z_e) \neq 0$, then $|f|\geq \dfrac{|f(z_e)|}{4}$ on a set with nonempty interior.
	\end{lemma}

\begin{proof}
	As in the proof of Lemma~\ref{lem:LCzero}, $f$ is possibly nonzero on six sets with nonempty interior --- the six listed in Lemma~\ref{lem:GrigInt}. Call the values on these sets $K_i$, $i = 1, \dots 6$.
	Writing $R:= f(z_e) = c_e$, we have
	\begin{eqnarray}
		R + c_b = K_1 && c_c + c_d = K_4\nonumber\\
		R + c_c = K_2 && c_b + c_d = K_5\label{eq:linearequations}\\
		R + c_d = K_3 && c_c + c_b = K_6\nonumber
	\end{eqnarray}
	Rearranging \eqref{eq:linearequations} we have
	\begin{align*}
	K_1 + K_2 - K_6 &= 2R\\
	K_2 + K_3 - K_4 &= 2R\\
	K_1 + K_3 - K_5 &= 2R
	\end{align*}
	Solving this linear system yields
	\begin{align*}
	K_1 &= -\frac12r + \frac12s + \frac12t + R\\
	K_2 &= \frac12r  -\frac12s + \frac12t + R\\
	K_3 &= \frac12r + \frac12s + -\frac12t + R\\
	K_4 &= r\\
	K_5 &= s\\
	K_6 &= t
	\end{align*}
	for $r,s,t\in \CC$.
	
	By way of contradiction, suppose that $|K_i|<\dfrac{|R|}{4}$ for $i = 1, \dots, 6.$
	Then in particular $|r|, |s|, |t| < \dfrac{|R|}{4}$ and so
	\[
	\left|-\frac12r + \frac12s +\frac12t\right|\leq \frac{|R|}{8} + \frac{|R|}{8} + \frac{|R|}{8} = \frac{3|R|}{8},
	\]
	which implies
	\begin{align*}
	|K_1| &= \left|R-\frac12r + \frac12s + \frac12t\right| \\
	&\geq \left||R| - \left|-\frac12r + \frac12s +\frac12t\right|\right|\\
	&= |R| - \left|-\frac12r + \frac12s +\frac12t\right|\\
	&\geq |R| - \frac{3|R|}{8} = \frac{5|R|}{8}
	\end{align*}
	which is a contradiction, since $|K_1|$ was supposed to be less than $\dfrac{|R|}{4}$.
	Hence $|K_i| \geq \dfrac{|R|}{4}$ for some $i$, and hence $|f| \geq \dfrac{|R|}{4}$ on a set with nonempty interior.
\end{proof}
\begin{lemma}\label{lem:cstarzeinterior}
	Suppose that $f\in B(\ggx)$, $f(z_e) \neq 0$, that $f_n\in A_\CC(\ggx)$ for all $n$ and that $f_n\to f$ uniformly.
	Then $\supp(f)$ has nonempty interior.
\end{lemma}
\begin{proof}
	Write $R:= f(z_e)$, and find $N$ such that $n\geq N$ implies $\|f - f_n\|_\infty< \dfrac{|R|}{10}$.
	Then in particular \[|f_N(z_e) - R| <\dfrac{|R|}{10} \quad  \text{ so } \quad |f_N(z_e)| \geq \dfrac{9|R|}{10}.\]
	Also, $0\notin B_{|R|/10}(f_N(z_e))$ .  If $V := f_N^{-1}(f_N(z_e))$ has nonempty interior then we would be done since $f(v) \in B_{|R|/10}(f_N(v)) = B_{|R|/10}(f_N(z_e))$ for all $v\in V$, implying that $f$ is nonzero on $V$.
	
	So suppose $V$ has empty interior.
	By the same reasoning as in the proof of Lemma~\ref{lem:zeinterior}, we can find $m\geq 0$ and $c_e, c_b, c_c, c_d\in \CC$ such that
	\[
	\left.f_N\right|_{U_m} = \sum_{g\in \{e, b, c, d\}}c_g 1_{U_{g,m}}
	\]
	where $c_e = f_N(z_e) = R$. Then Lemma~\ref{lem:LCboundedbelow} implies that there exists a set $W$ with nonempty interior such that for all $w\in W$ we have \[|f_N(w)| \geq \dfrac{|f_N(z_e)|}{4}\geq \dfrac{9|R|}{40} .\]
	Then for all $w\in W$, $|f(w) - f_N(w)|< \dfrac{|R|}{10}$ implies that $|f(w)|$ is at least
	 \[\dfrac{9|R|}{40} - \dfrac{|R|}{10} = \dfrac{5|R|}{40}\] for all $w\in W$.
	Hence $f$ is nonzero on a set with nonempty interior.	
\end{proof}
Now, as in Lemma~\ref{lem:grigsingular}, we have the same conclusion for general nonzero uniform limits of elements of $A_\CC(\ggx)$ by using Lemma~\ref{lem:cstarzeinterior} and translating.
\begin{lemma}\label{lem:limitinterior}
	Suppose that $0 \neq f \in B(\ggx)$ and that $f_n\to f$ uniformly with $f_n\in A_\KK(\ggx)$ for all $n$.
	Then $\supp(f)$ has nonempty interior.
\end{lemma}
\begin{proof}
	Find $\gamma\in\ggx$ and as before set $R: = f(\gamma)$.
	Again find $N$ such that for all $n\geq N$ we have $\|f_n - f\|_\infty< \dfrac{|R|}{10}$ --- for the same reasons as in the proof of Lemma~\ref{lem:cstarzeinterior} we can assume that $f_N^{-1}(f_N(\gamma))$ has empty interior.
	
	As before, it follows from this that there exists a compact bisection $D$ of the form $\Theta((\alpha, g, \beta), C(\beta\eta))$ such that $\gamma\in \overline D \setminus D$.
	By Lemma~\ref{lem:ssgcalc} this implies that there exists $s\in \sgx$ and $\mu\in X^*$ such that $\gamma = [s, \mu1^\infty]$.
	As in the proof of Lemma~\ref{lem:grigsingular}, letting
	\[
	\xi = [(\mu, e, 1^{|\mu|}), 1^\infty]
	\]
	yields $\xi^{-1}\gamma^{-1}\gamma\xi = z_e$, and finding compact open bisections $B$ and $C$ with $(\gamma\xi)^{-1}\in B$, $\xi\in C$ gives
	\begin{eqnarray*}
		1_B * f_N* 1_C (z_e) &=& 1_B* f_N* 1_C (\xi^{-1}\gamma^{-1}\gamma\xi)\\
		&=& f_N* 1_C(\gamma\xi)\\
		&=& f_N(\gamma) \neq 0.
	\end{eqnarray*}
	Let $g: = 1_B * f_N *1_C$.
    Then as in the proof of Lemma~\ref{lem:cstarzeinterior} there exists a set $W$ with nonempty interior such that for all $w\in W$ we have $|g(w)| \geq \dfrac{|f_N(\gamma)|}{4}\geq \dfrac{9|R|}{40} $. For any bisection $D$ and any $h\in B(\ggx)$, the ranges
of $1_D * h$ and $h*1_D $ are contained in the range of $h$, so $\|1_D * h\|_\infty, \|h*1_D \|_\infty \leq \|h\|_\infty$.
	Hence
	\[
	\|g - 1_B*f*1_C\|_\infty \leq \|f_N - f\|_\infty< \dfrac{|R|}{10}	
	\]
	and so $1_B*f*1_C$ is nonzero on a set with nonempty interior.
	But then the support of $f$ has nonempty interior, because otherwise the same would be true of $1_B*f*1_C$.
\end{proof}
We can now complete the proof of Theorem~\ref{thm:grigsimple}.
 \begin{proof}[Proof of Theorem~\ref{thm:grigsimple}\eqref{it:grigthm2}]
	This follows from Theorem~\ref{thm:c*simple}(3) , Lemma~\ref{lem:limitinterior}, the fact that amenability of $G$ implies amenability of $\ggx$ by \cite[Corollary 10.18]{EP17}.
\end{proof}



\begin{thebibliography}{99}

\bibitem{Anderson} J. Anderson, \emph{Extensions, restrictions, and representations of
    states on $\cs$-algebras}, Trans. Amer. Math. Soc. \textbf{249} (1979), 303--329.

\bibitem{BCFS}  J.H. Brown, L.O. Clark, C. Farthing and A. Sims, \emph{Simplicity of
    algebras associated to \'etale groupoids}, Semigroup Forum \textbf{88} (2014),
    433--452.

\bibitem{BNRSW} J.H. Brown, G. Nagy, S. Reznikoff, A. Sims and
D.P. Williams, \emph{Cartan subalgebras in $\cs$-algebras of Hausdorff \'etale groupoids}, Integral Eq. Oper. Th. \textbf{85} (2016), 109--126.

\bibitem{CEP} L.O. Clark, R. Exel and E. Pardo,
\emph{A Generalised uniqueness theorem and the graded ideal structure of Steinberg algebras}, Forum Math. \textbf{30} (2018), 533--552.

\bibitem{CFST} L.O. Clark, C. Farthing, A. Sims and M. Tomforde, {\em A groupoid generalization of Leavitt path algebras},
Semigroup Forum \textbf{89} (2014), 501--517.

\bibitem{CS2015} L.O. Clark and A. Sims, \emph{Equivalent groupoids have Morita equivalent Steinberg algebras}, J. Pure Appl. Algebra \textbf{219} (2015), 2062--2075.

\bibitem{Connes82} A. Connes,
\emph{A survey of foliations and operator algebras},
in: Operator Algebras and Applications, \emph{Proc.
Symp. Pure Math. A.M.S.} \textbf{38} part I (1982), 521--628.


\bibitem{Connes}A. Connes, ``Non commutative geometry'', Academic Press, 1994.



\bibitem{Exel} R. Exel, \emph{Non-Hausdorff \'etale groupoids}, Proc. Amer. Math. Soc. \textbf{139} (2011), 897--907.

\bibitem{Exel2} R. Exel,  \emph{Reconstructing a totally disconnected groupoid from its ample semigroup},
Proc. Amer. Math. Soc.  \textbf{138} (2010), 2991--3001.

\bibitem{Exel3} R. Exel,  \emph{Inverse semigroups and combinatorial $\cs$-algebras}, 	
Bull. Braz. Math. Soc. \textbf{39} (2008), 191--313.

\bibitem{EP16}
R. {Exel} and E. {Pardo}, \emph{{The tight groupoid of an inverse
  semigroup}}, Semigroup Forum \textbf{92} (2016), 274--303.

\bibitem{EP17}
R. Exel and E. Pardo, \emph{Self-similar graphs, a unified treatment of
  Katsura and Nekrashevych C*-algebras}, Adv. Math. \textbf{306}
  (2017), 1046--1129.

\bibitem{ExelPitts}
R. Exel and D. Pitts, \emph{Weak Cartan inclusions and non-Hausdorff groupoids}, preprint, 2019,
arXiv:1901.09683 [math.OA].



\bibitem{GivHal} St. Givant and P. Halmos,
``Introduction to Boolean algebras'', Undergraduate Texts in Mathematics.
Springer, New York, 2009. xiv+574 pp.

\bibitem{Gr80}
R. I. Grigorchuk, {\em Burnside problem on periodic groups}, Funktsional. Anal. i Prilozhen., {\bf 14} (1980), 53--54; Funct. Anal. Appl., {\bf 14} (1980), 41--43.

\bibitem{Gr84}
R. I. Grigorchuk.
\newblock {\em Degrees of growth of finitely generated groups, and the theory of
invariant means},
Mathematics of the USSR-Izvestiya, {\bf25}(2) (1985) 259.

\bibitem{Ka08}
T. Katsura,
{\em A construction of actions on {K}irchberg algebras which induce given
actions on their {K}-groups},
J. reine angew. Math {\bf 617} (2008), 27--65.

\bibitem{KhoshSkand}
M. Khoshkam and G. Skandalis,
\emph{Regular representations of groupoid $\cs$-algebras and applications to inverse semigroups},
J. reine angew. Math \textbf{546} (2002), 47--72.

\bibitem{LRRW14}
M. Laca, I. Raeburn, J. Ramagge, and M.~F. Whittaker,
  \emph{Equilibrium states on the Cuntz–Pimsner algebras of self-similar
  actions}, J. Funct. Anal. \textbf{266} (2014), 6619 -- 6661.

\bibitem{Nek09}
V. Nekrashevych, \emph{C*-algebras and self-similar groups}, J. reine
  angew. Math \textbf{630} (2009), 59--123.

\bibitem{Nek16}
V. Nekrashevych,
\newblock {\em Growth of \'etale groupoids and simple algebras},
\newblock Internat. J. Algebra Comput. \textbf{26}(2) (2016), 375--397.

\bibitem{Phillips} N.C. Phillips, \emph{Crossed products of the Cantor set by free
    minimal actions of $\Z^d$}, Comm. Math. Phys. \textbf{256} (2005), 1--42.

\bibitem{Raeburn} I. Raeburn, ``Graph Algebras'', CBMS Reg. Conf. Ser. Math., vol. 103, Amer. Math.
Soc., Providence, RI, 2005.

\bibitem{Ren} J. Renault, ``A Groupoid Approach to {$C^*$}-Algebras'',
    Springer, Berlin, 1980.

\bibitem{Ren2} J. Renault, \emph{Cartan subalgebras in $\cs$-algebras}, Irish Math. Soc. Bulletin
\textbf{61} (2008), 29--63.

\bibitem{Sta15}
C. Starling,
  \emph{Boundary quotients of {C}*-algebras of right {LCM} semigroups}, J. Funct. Anal. \textbf{268} (2015), 3326--3356.

\bibitem{St2} B. Steinberg, \emph{Simplicity, primitivity and semiprimitivity of \'etale groupoid algebras with
applications to inverse semigroup algebras}, J. Pure Appl. Algebra \textbf{220} (2016), 1035--1054.

\bibitem{St}  B. Steinberg, {\em A groupoid approach to inverse semigroup algebras}, Adv. Math. {\bf 223} (2010), 689--727.



\end{thebibliography}
\end{document}